\documentclass[11pt, a4paper, reqno]{amsart}
\usepackage{amsmath, amsthm, amssymb, dsfont, color, fancyhdr, graphicx, enumerate,  mathabx, comment}
 \usepackage{setspace}
 \usepackage{youngtab}
 \usepackage{subfigure}
 \includecomment{todo}

\usepackage[margin=3.2cm]{geometry}
\usepackage[arrow, matrix, curve]{xy}
\allowdisplaybreaks

\theoremstyle{plain}
\newtheorem{theorem}{Theorem}[section]
\newtheorem*{theorem*}{Theorem}
\newtheorem{lem}[theorem]{Lemma}
\newtheorem{prop}[theorem]{Proposition}

\newtheorem{cor}[theorem]{Corollary}

\theoremstyle{definition}

\newtheorem{ex}[theorem]{Example}
\newtheorem{defn}[theorem]{Definition}
\newtheorem{rk}[theorem]{Remark}

\title[\tiny On the modular representation theory of the partition algebra]{On the modular representation theory of the partition algebra}

\author{Armin Shalile}

\address{Institute for Algebra und Number Theory, University of Stuttgart,
Pfaffenwaldring 57, 70569 Stuttgart, Germany.}

\date{\today}

\email{shalile@mathematik.uni-stuttgart.de}

\newcommand{\C}{\mathbb{C}}

\newcommand{\Z}{\mathbb{Z}}
\newcommand{\R}{\mathbb{R}}
\newcommand{\N}{\mathbb{N}}

\newcommand{\inv}{^{-1}}

\newcommand{\prd}{P_r (\delta)}
\newcommand{\prdh}{P_{r+1/2} (\delta)}
\newcommand{\arc}{\ar@{-}@/_/}  
\newcommand{\tra}{\ar@{-}}      
\newcommand{\dta}{\ar@{.}}      
\newcommand{\ltra}{\ar@{<-}}      
\newcommand{\rtra}{\ar@{->}}       
\newcommand{\larc}{\ar@{<-}@/_/}  
\newcommand{\rarc}{\ar@{<-}@/_/}  

\newcommand{\Hom}{\text{Hom}}

\newcommand{\iind}{\operatorname{i-ind}}
\newcommand{\ires}{\operatorname{i-res}}
\newcommand{\ind}{\text{ind}}

\newcommand{\uqslp}{\mathcal{U}_q(\widehat{\mathfrak{sl}_p})}
\newcommand{\uqslinf}{\mathcal{U}_q(\mathfrak{sl}_\infty)}

\newcommand{\res}{\text{res}}

\newcommand{\dmu}{\Delta(\mu)}
\newcommand{\dmup}{\Delta(\mu')}
\newcommand{\dla}{\Delta(\lambda)}
\newcommand{\dlap}{\Delta(\lambda')}
\newcommand{\lmu}{L(\mu)}
\newcommand{\lla}{L(\lambda)}
\newcommand{\lmup}{L(\mu')}

\newcommand{\UD}{\textrm{Tab}}
\newcommand{\wt}{\operatorname{wt}}

\newcommand{\xc}{\underset{\textstyle\bigwedge}{\bigvee}}

\newcommand{\vc}{\underset{\textstyle\bigcirc}{\bigvee}}

\newcommand{\wc}{\underset{\textstyle\bigwedge}{\bigcirc}}

\newcommand{\oc}{\underset{\textstyle\bigcirc}{\bigcirc}}

\newcommand{\qc}{\underset{\textstyle \raisebox{-1.9mm}{$\bigcirc$}}{?}}

\begin{document}
\setlength\abovedisplayskip{6pt}
\setlength\belowdisplayskip{6pt}
\begin{abstract} 
We determine the decomposition numbers of the partition algebra when the characteristic of the ground field is zero or larger than the degree of the partition algebra. This will allow us to determine for which exact values of the parameter the partition algebras are semisimple over an arbitrary field. Furthermore, we show that the blocks of the partition algebra over an arbitrary field categorify weight spaces of an action of the quantum groups $\uqslp$ and $\uqslinf$ on an analogue of the Fock space. In particular, we recover the block structure which was recently determined by Bowman et al. In order to do so, we use induction and restriction functors as well as analogues of Jucys-Murphy elements. The description of decomposition numbers will be in terms of combinatorics of partitions but can also be given a Lie theoretic interpretation in terms of a Weyl group of type $A$: A simple module $L(\mu)$ is a composition factor of a cell module $\Delta(\lambda)$ if and only if $\lambda$ and $\mu$ differ by the action of a transposition.
\end{abstract}

\maketitle

\section*{Introduction}
The partition algebra was introduced by Paul Martin \cite{Martin94} for the study of the Potts model in statistical mechanics and independently by Jones \cite{Jones94}. It possesses a diagrammatically defined basis and multiplication and is therefore a diagram algebra, a class of algebras which for example also encompasses the group algebra of the symmetric group, the Brauer algebra and the Temperley-Lieb algebra to name but a few. A typical property for these diagram algebras is that they are related to subgroups of the general linear group via Schur-Weyl duality. This duality was first observed by Schur \cite{Schur} and classically relates the representation theory of the general linear group and the symmetric group. Both act on the $r$-fold tensor product $E^{\otimes r}$ of the natural representation $E$ of the general linear group: The general linear group by diagonal extension of its action on $E$ and the symmetric group (on $r$ letters) by permuting the tensor components. The resulting bimodule structure on $E^{\otimes r}$ can be used to define functors between the respective module categories, the Schur functor and inverse Schur functor.  The partition algebra is also in Schur-Weyl duality to the symmetric group. However, here we view the symmetric group as embedded in the general linear group as permutation matrices, so the symmetric group replaces the general linear group and the partition algebra replaces the symmetric group in the classical setting.

Many authors have studied partition algebras. Martin \cite{Martin96} determined much of its structure over $\C$ by determining criteria for semisimplicity and the structure of the generically irreducible representation when these are non-semisimple. Xi \cite{Xi99} showed that they are cellular in the sense of \cite{GrahamLehrer} and it was shown in \cite{4authors} that they are even cellularly stratified. Recently, Bowman et al. \cite{BDK} determined a criterion for when two simple modules are in the same block over an arbitrary field. 

Partition algebras have a basis of diagrams with $r$ columns where $r$ is an integer called the degree of the partition algebra and are denoted by $\prd$ in this paper where $\delta$ is some parameter. The partition algebra of degree $r$ can be embedded in a partition algebra of larger degree and the interdependency of this family may be analyzed by a structure called a tower of recollement \cite{towersofrecollment} or alternatively in the formalism of cellularly stratified algebras as introduced in \cite{4authors}. For the study of Jucys-Murphy elements and restriction rules it does, however,  turn out that it is easier to study the partition algebra by adding to this family certain subalgebras $P_{r-1/2}(\delta)$ of the usual partition algebra \cite{Martin00}. These are again diagrammatically defined, form a tower of recollement and are cellularly stratified.

The aim of this paper is to determine the multiplicities of the simple modules in the cell modules, the so called decomposition numbers, over a field of characteristic $p$ which does not divide the degree of the partition algebra, that is, $p=0$ or $p>r$. As a corollary, we recover the explicit condition when two simple modules are contained in the same block over a field of arbitrary characteristic from \cite{BDK}. Notice that  over the complex numbers, both of these results are known by the work of Martin \cite{Martin96}. The description of decomposition numbers will be in terms of combinatorics of partitions but we will also show how it can be interpreted in a Lie theoretic way by the action of an (affine) Weyl group of type A. For blocks, a Lie theoretic description was known before by the work of Bowman, De Visscher and King for partition algebra and Cox, Martin and De Visscher for Brauer algebras over $\C$, see \cite{BDK} \cite{CDM2}. For decomposition numbers of diagram algebras, however, this is a new phenomenon.

The approach of this paper treats the case of finite characteristic and characteristic $0$ in a uniform manner and is motivated by the Okounkov-Vershik approach to the representation theory of the symmetric group. In \cite{OkounkovVershik}, Okounkov and Vershik propose to study the complex representation theory of the symmetric group by using only restriction rules and the action of a family of distinguished elements, the  Jucys-Murphy elements. This has also proved an effective tool in the study of the blocks of the symmetric group and the construction of central idempotents \cite{Murphy}. A similar approach was used by the author to study the decomposition numbers of Brauer algebras in non-dividing characteristic \cite{JMpaper}. Jucys-Murphy elements have been defined by Ram and Halverson for the partition algebra in \cite{HalversonRam} where their  action on the cell modules was also studied. Enyang \cite{Enyang13} also defined seminormal forms for partition algebras.

The paper proceeds as follows. After recalling some well-known properties of partition algebras and Jucys-Murphy elements, we extend the restriction rules in the non-semisimple case determined by Doran and Wales \cite{DW2000} to the restriction between full and half integer partition algebras, see Theorem \ref{thm restriction of cells}. In the semisimple case this is already known by work of Martin \cite{Martin00}.

We then introduce the combinatorial framework needed in this paper. We introduce an arrow diagram calculus which is inspired by, but in many ways different to,  the arrow diagram calculus introduced by Brundan and Stroppel in the study of Khovanov diagram algebras \cite{BrundanStroppel} and used by Cox and De Visscher for the determination of the decomposition numbers of Brauer algebras over $\C$ \cite{CoxDeVisscher}. We will show that this arises naturally  out of the action of the Jucys-Murphy elements on the cell modules.

We proceed by proving that if a simple module $L(\mu)$ is a composition factor of a cell module $\dla$, then both must contain a simultaneous eigenvector for the action of the Jucys-Murphy elements and the eigenvalue by which a Jucys-Murphy element acts on these vectors is the same for both modules. We say that $\lambda$ and $\mu$ have a common JM weight in this case. We then prove that the converse is also true:
\setcounter{section}{4}
\setcounter{theorem}{0}
\begin{theorem} Consider the partition algebra $P_d(\delta)$ with $d \in \{ r,r+1/2 \}$ over a field of characteristic $p=0$ or $p>r$ and let $\mu$ be a partition of $r$. Then $\lambda$ and $\mu$ have a common JM weight if and only if $L(\mu)$ is a composition factor of $\dla$.
\end{theorem}
\setcounter{section}{0}
\setcounter{theorem}{0}
We will show that this actually does suffice to determine all decomposition numbers since all decomposition numbers are either $0$ or $1$ (Corollary \ref{cor dec multiplicities are 0 or 1}).

In order to prove the theorem, we first show that if $\lambda$ and $\mu$ have a common JM weight, then they are only different in one row, see Proposition \ref{prop common jm weight implies certain form}.

We then introduce refined induction and restriction functors in Section 3 as a preparation for an inductive proof of Theorem \ref{thm determination of dec nos}. These are analogues of the well-known $i$-restriction and $i$-induction functors for symmetric groups. We explicitly determine the $i$-restriction of simple modules except in one case where we can only describe the socle explicitly.

In Section 4, we then add the ingredients from the previous sections together to prove Theorem \ref{thm determination of dec nos}. Together with Proposition \ref{prop common jm weight implies certain form} this gives a diagrammatic condition to read off decomposition numbers straight from the partitions in question. It turns out that the diagrammatic condition can be described by the action of a Weyl group of type $A$ and the condition reduces to the partitions being connected by the action of a transposition, see Theorem \ref{thm lie theoretic det of dec nos}. Theorem \ref{thm determination of dec nos} will also allow us to state a precise condition for semisimplicity over arbitrary fields (Theorem \ref{thm semisimplicity}).

Section 5 prepares the categorification result in Section 6. We study the blocks of the partition algebra over a field of arbitrary characteristic which we completely determine in Theorem \ref{thm determination of blocks}. For this we use a result of Mathas \cite{Mathaspaper} which gives a necessary condition for two simple module to belong to the same block of an arbitrary cellular algebra with Jucys-Murphy elements. We show that this condition is also sufficient by transferring the knowledge of the blocks in characteristic $0$ to characteristic $p$ by reduction modulo $p$. We also recover the Lie theoretic description of blocks from \cite{BDK} as orbits of the action of an (affine if $p>0$) Weyl group of type $A$ in a natural way, see Theorem \ref{thm lie determination of blocks}.

In Section 6, we define an analogue of the Fock space for partition algebras. The definition is set-up in such a way that we obtain an action of the quantum group $\uqslp$ or $\uqslinf$ (depending on whether $p>0$ or $p=0$) by mimicking the action of the $i$-induction and $i$-restriction functors, see Theorem \ref{thm fock space is uqslp module}. The theorem is an analogue of a similar result originally due to Hayashi \cite{Hayashi} and Misra and Miwa \cite{MisraMiwa} for Hecke algebras of type $A$. There are two major differences to the Hecke algebra case though: Firstly, in order to get the right action, we need to adjust the induction/restriction functors slightly by composing with a functor due to Martin connecting half integer and full integer partition algebras. This will require a shift in the parameter $\delta$ so that our Fock space is a direct sum of the module categories of $P_i(\delta+i)$ for $i \in \N$. Secondly, the induction functor is not exact in our case, so that we cannot simply define the action by using the map induced by passing to Grothendieck groups. Nevertheless, the weight spaces correspond precisely to blocks, see Theorem \ref{thm decat of block is weight space}.

We would like to point out that at the time of writing this paper, it was brought to our attention that the determination of decomposition numbers in non-dividing characteristic were also obtained independently and to our knowledge by different methods by O. King \cite{King}. 

\section{Preliminaries}
\subsection{Definition of partition algebras} Throughout this paper, let $F$ be a field of characteristic $p \geq 0$, $r$ a natural number and let $d \in \{r,r+1/2 \}$. We also fix a parameter $\delta \in F$. The partition algebra $\prd$ of degree $r$ with parameter $\delta$ has a basis consisting of set partitions of the set $\{1,1',2,2',\ldots,r,r' \}$ which are visualized diagrammatically as partition diagrams. A partition diagram is a $2 \times r$ array of dots, which are labelled $1,\ldots,r$ in the top row and $1',\ldots, r'$ in the bottom row. We connect dots by edges in such a way, that the connected components are precisely the blocks of a partition and a minimal number of edges is used. Notice that this is not unique. For example, the diagram 
$\begin{array}{c}
 \begin{xy}
\xymatrix@!=0.01pc{ \bullet & \bullet \tra[drrrr] & \bullet  \arc[ll] \tra[r]  & \bullet \tra[dr] & \bullet \tra[r] & \bullet \\ \bullet \tra[r] & \bullet \tra[urrrr] & \bullet \arc[rr] & \bullet & \bullet & \bullet \\ }
\end{xy}\end{array}$
represents the set partition $\{  \{ 1,3,4,3',5' \}, \{2,6' \}, \{4'\}, \{5,6,1',2' \} \}$. The number of blocks of the partition which contain both an element of $\{1,2,\ldots,r\}$ as well as an element of $\{1',2',\ldots,r'\}$ is called the propagating number. The propagating number could also be defined as the minimum number of edges which connect top and bottom row ("propagating lines") in the diagram of the partition. In the example above, the propagating number is $3$. 

We can define a multiplication on partitions diagrammatically by a process called concatenation. To concatenate two partition diagrams $a$ and $b$, we write $a$ on top of $b$ and identify adjacent rows. The concatenation $a \circ b$ is the diagram obtained from this construction by deleting all connected components which are not connected to the top or bottom row and premultiplying the resulting diagram with as many powers of the parameter as there are connected components deleted. An example of this process is given in Figure \ref{figure example of concatenation}.

\begin{figure}
\caption{An example of concatenation}
\label{figure example of concatenation}
\vspace{3pt}
 $\begin{array}{c}\begin{xy}
\xymatrix@!=0.01pc{ 
\bullet \tra[d] & \bullet & \bullet \arc[ll] & \bullet \arc[rr] & \bullet  \arc[ll] & \bullet  \\
\bullet \dta[d]  \ar@{}[u]^{\text{ \large $a =$ \  }}  & \bullet \dta[d] & \bullet \dta[d]  & \bullet \arc[ll] \dta[d]  & \bullet \dta[d] & \bullet  \dta[d] \ar@{}[d]^{\text{ \large \ $\Longrightarrow \ a \circ b = $   }}  \\ 
\bullet \tra[rd] & \bullet  & \bullet \tra[d] & \bullet \arc[rr] &\bullet &  \bullet   \\
\bullet \arc[rr] \ar@{}[u]^{\text{ \large $b =$ \  }}   &\bullet & \bullet & \bullet &
\bullet & \bullet \tra[l]}
\end{xy}
 \begin{xy}
\xymatrix@!=0.01pc{ \ & \ &\ & \ \\
\bullet \tra[rd]  & \bullet & \bullet \arc[ll]  & \bullet \arc[rr] & \bullet \arc[ll] & \bullet      \\ 
\bullet   \arc[rr] \ar@{}[u]^{\text{ \large $\delta^2 \cdot $ \  }} & \bullet  & \bullet & \bullet &\bullet \tra[r] &  \bullet  \\
\ & \ &\ & \ }
\end{xy} \end{array}.  $

\end{figure}
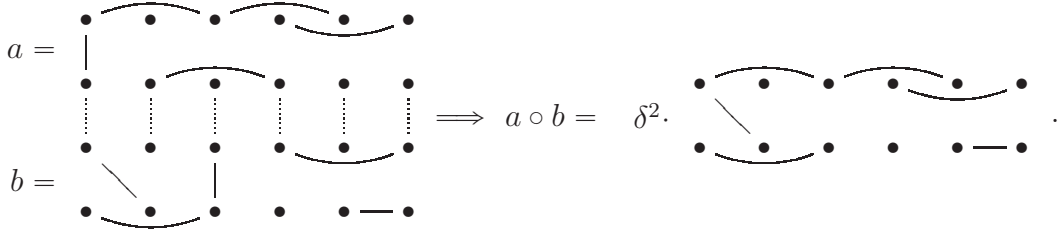

The partition algebra $\prd$ has a subalgebra $P_{r-1/2}(\delta)$ which is spanned by all diagrams where $r$ and $r'$ belong to the same block of the partition. The algebra $P_{r-1/2}(\delta)$ in turn contains the partition algebra $P_{r-1}(\delta)$ as a subalgebra where we identify $P_{r-1}(\delta)$ with the subalgebra spanned by all partitions in which $\{r,r'\}$ is a block. Thus we get a chain of subalgebras $$P_{1/2}(\delta) \subset P_{1}(\delta) \subset P_{3/2}(\delta) \subset P_{2}(\delta) \subset \ldots \subset \prd.$$

The half integer degree partition algebras are in fact Morita equivalent to full integer degree partition algebras:

\begin{theorem}[\cite{Martin00}]
\label{thm morita equiv of half with full}
There is an equivalence $T$ between between the module categories of $P_{r}(\delta)$ and $P_{r-1/2}(\delta-1)$.
\end{theorem}

We will write $P_d(\delta)$ with $d \in \{r,r+1/2 \}$ to cover both cases. We will also need the following distinguished elements of the partition algebra:
\begin{align*}
s_i & = \{ \{1,1' \},\{2,2' \},\ldots,\{i,(i+1)' \}, \{i+1,i' \},\{i+2,(i+2)' \},\ldots, \{r,r' \} \} \\
A_i& =\{ \{1,1' \},\{2,2' \},\ldots,\{i\}, \{ i' \}, \{i+1,(i+1)' \},\ldots, \{r,r' \} \}\\
A_{ij}& =\{ \{1,1' \},\{2,2' \},\ldots,\{i,i+1,i',(i+1)' \}, \{i+2,(i+2)' \},\ldots, \{r,r' \} \}
\end{align*}
 which may be visualized as follows:

$$
 s_i = \begin{array}{c}
 \begin{xy}
\xymatrix@!=0.01pc{ 
\bullet \tra[d]^{ \hspace*{0.2cm} \cdots} &  \bullet \tra[d]	 &	\overset{i}{\bullet}  \tra[rd] & \overset{i+1}{\bullet} \tra[ld] &	\bullet \tra[d]^{ \hspace*{0.2cm} \cdots} & \bullet \tra[d] \\
\bullet &  \bullet	 &\bullet & \bullet &	\bullet & \bullet
}
\end{xy}
\end{array},
 A_i = \begin{array}{c}
 \begin{xy}
\xymatrix@!=0.01pc{ 
\bullet \tra[d]^{ \hspace*{0.2cm} \cdots} &  \bullet\tra[d]	 &	\overset{i}{\bullet}  & \bullet \tra[d]^{ \hspace*{0.2cm} \cdots} &	\bullet \tra[d]  \\
\bullet &  \bullet	 &\bullet& \bullet &	\bullet 
}
\end{xy}
\end{array},$$ $$ A_{i,i+1} = \begin{array}{c}
 \begin{xy}
\xymatrix@!=0.01pc{ 
\bullet \tra[d]^{ \hspace*{0.2cm} \cdots} &  \bullet \tra[d]	 &	\overset{i}{\bullet}  \tra[r] & \overset{i+1}{\bullet} \tra[d]  &	\bullet \tra[d]^{ \hspace*{0.2cm} \cdots} & \bullet \tra[d] \\
\bullet &  \bullet	 &	\bullet \tra[r] & \bullet &	\bullet & \bullet
}
\end{xy}
\end{array}. $$

The partition algebra is generated as an algebra by the elements $s_i, A_i$ and $A_{i,i+1}$ for $i=1,\ldots,r-1$. The subalgebra generated by the elements $s_i$ ($i=1,\ldots,r-1$) is precisely the symmetric group algebra $FS_r$. 

\subsection{Simple modules and cellularity}
The  generic irreducible representations of the partition algebra have been studied by many authors and, accordingly, there are quite a few realizations around. We will use different realizations in different contexts in order to shorten proofs at the expense of elegance. Many  structural properties of partition algebras can be derived from the more general construction of cellularly stratified algebras in the sense of \cite{4authors}. We need the following definition:

\begin{defn} \label{defn of the stratification data}
\begin{enumerate}[(a)]
\item Let $J_{l,r}$ be the ideal of the partition algebra $P_r(\delta)$ which is spanned by all diagrams on $r$ dots with at most $l$ propagating lines. Similarly, let $J_{l+1/2,r+1/2}$ be the ideal spanned by all diagrams on $r+1$ dots with at most $l+1$ propagating lines, for which $r+1$ and $(r+1)'$ belong to the same block of the diagram. We usually omit $r$ if it can be recovered from the context in this and the following definitions.
\item Define $I_{l}=J_{l}/J_{l-1}$ and $I_{l+1/2}=J_{l+1/2}/J_{l-1/2}$ which are spanned by diagrams with exactly $l$ and $l+1$ propagating lines, respectively.
\item For $l=0,1,\ldots,r$, define $e_l \in P_{d}(\delta)$ to be the diagram 
$$ e_l = \begin{array}{c}
 \begin{xy}
\xymatrix@!=0.01pc{ 
	\bullet \tra[d] \tra[r] & \bullet \tra[r]  &	\cdots \tra[r] & \bullet & \overset{r-l+1}{\bullet} \tra[d]^{ \hspace*{0.2cm} \cdots} &  \bullet \tra[d]	 \\
	\bullet \tra[r] & \bullet \tra[r] &	\cdots \tra[r] & \bullet & \bullet &  \bullet	
}
\end{xy}
\end{array}, e_0 = \frac{1}{\delta} \begin{array}{c}
 \begin{xy}
\xymatrix@!=0.01pc{ 
\bullet \tra[r]  & \bullet \tra[r] &  \bullet \tra[r] &  \cdots \tra[r]  &  \bullet \tra[r] &  \bullet \tra[r]  &  \bullet  \\
\bullet \tra[r]  & \bullet \tra[r] &  \bullet \tra[r] &  \cdots \tra[r]  &  \bullet   \tra[r] &  \bullet \tra[r]  &  \bullet  \\
}
\end{xy}
\end{array}$$ where $e_0$ is only defined in the case $d=r$ and if $\delta \neq 0$. 
\end{enumerate}
\end{defn}

To show that the partition algebra is cellularly stratified, we need to introduce certain vector spaces $V_l$ and $V_{l+1/2}$ which will be in the spirit of Section 2 of \cite{DW2000}. First define a vector space $V'_l$ by the linear span of diagrams of the following form: The first $r-l$ dots in the bottom row are singletons, while each $i \in \{(r-l+1)',(r-l+2)',\ldots,r'\}$ is in a block $b_i$ with $\emptyset \neq b_i \setminus \{ i'\} \subseteq \{1,\ldots,r \}$, that is,  each is connected to the top row but no two of them to the same block. $V'_{l+1/2}$ is spanned by the same diagrams as $V_{l+1}$ except that we require $r+1$ and $(r+1)'$ to be connected.

The symmetric group $S_l$  acts on $V'_l$ and $V'_{l+1/2}$ from the right by permuting the positions $r-l+1,r-l+2,\ldots,r$. Notice that both spaces are also naturally left $P_r(\delta)$- and $P_{r+1/2}(\delta)$-modules, respectively. We define the vector spaces $V_l$ and $V_{l+1/2}$ to be spanned by a set of representatives of the orbits of the action of $S_l$ by taking all diagrams in $V'_l$ and $V'_{l+1/2}$, respectively, with the following property: If $i<j \leq l$ are natural numbers and $b_i$ and $b_j$ are the blocks connected to $i'$ and $j'$, respectively, then $b_i$ is further to the left than $b_j$, that is, $\min (b_i \setminus \{i'\}) < \min (b_j \setminus \{j'\})$. This ensures that we can draw the propagating lines without crossings.

This definition of $V_l$ is essentially the same as the one given by Xi in Section 4 of \cite{Xi99}.
 Thus, $P_r(\delta)$ is an iterated inflation of symmetric group algebras $FS_l$ along the vector spaces $V_l$ for $l=0,1,\ldots,r$ and exactly in the same way one shows that $P_{r+1/2}(\delta)$ is an iterated inflation of $FS_l$ along the vector spaces $V_{l+1/2}$. In particular, we obtain:

\begin{theorem}[See also Proposition 2.6 in \cite{4authors}] \begin{enumerate}[(a)]
\item Given an arbitrary field $F$ and $\delta \neq 0$, the partition algebra $P_d(\delta)$ is cellularly stratified with stratification data given by $(FS_r,V_d,FS_{r-1}, V_{d-1},\ldots,FS_0,V_{0}/V_{1/2})$ and idempotents $e_l$ for $l=0,\ldots,r$.
\item If $\delta=0$, then $P_d(0)/J_0$ is cellularly stratified with stratification data given by $(FS_r,V_d,FS_{r-1}, V_{d-1},\ldots,FS_1,V_{1}/V_{3/2})$ and idempotents $e_l$ for $l=1,\ldots,r$.
\end{enumerate}
\end{theorem}

Cellularly stratified algebras naturally come with functors which in our case relate representations of the symmetric group algebras $FS_l$ with representations of $P_d(\delta)$:

\begin{defn}[\cite{4authors}, Definition 3.3]
For $l=0,1,\ldots,r$, we set $ B_l=$ \\ $P_d(\delta)e_l \otimes_{e_lP_d(\delta)e_l}FS_l$ and define the layer induction functor $G_{l,d}$, or just $G_l$ if $d$ is clear from the context, as \begin{align*}
G_{l,d}: FS_l \text{-mod} & \to P_{d}(\delta)\text{-mod} \\ M & \mapsto B_l \otimes_{e_lP_d(\delta)e_l} M.
\end{align*}
\end{defn}

One can also construct cell modules from the stratification data. We first need to recall some definitions: A partition of a natural number $n$ is a tuple $ \lambda=(\lambda_1,\lambda_2, \ldots, \lambda_k)$ such that $\lambda_i \in \N$, $\lambda_i \geq \lambda_{i+1}$ for $i=1,\ldots,k-1$ and $\sum_{i=1}^k \lambda_i=n$. Each $\lambda_i$ is called a part of $\lambda$ and the natural number $k$ is called the number of parts of $\lambda$. We occasionally view the partition $\lambda$ as $(\lambda_1,\lambda_2, \ldots, \lambda_k,0,0,\ldots)$. We consider $\lambda = \emptyset$ the unique partition of $0$. We denote by $\Lambda^+(r)$ the set of partitions of non-negative integers $l \leq r$. The set $\Lambda^+(r+1/2)$ is equal to $\Lambda^+(r)$ but will be treated formally different: they will correspond to the labelling set of simple modules of $\prd$ and $\prdh$. We use $\Lambda^+(d)$ for either  $\Lambda^+(r+1/2)$  or $\Lambda^+(r)$. Given a prime $p$, the partition $\lambda$ is called $p$-singular if there is $j \leq k-p+1$ such that $\lambda_j=\lambda_{j+1}=\lambda_{j+p-1}$ and $p$-regular otherwise. For $p=0$ all partitions are considered $p$-regular.

We will usually identify a partition with its Young diagram. Thus, we will speak of addable and removable boxes of a partition: These are boxes of the Young diagram of $\lambda$ such that adding or removing this box will still yield a partition. If the box is $\epsilon$, then we denote the newly created partition by $\lambda+\epsilon$ and $\lambda-\epsilon$, respectively. We denote by  $\textrm{add } \lambda$ and $\textrm{rem }  \lambda$ the set of addable and removable boxes of $\lambda$.

In the following theorem, denote by $S(\lambda)$ the Specht module corresponding to the partition $\lambda$. The Specht modules $S(\lambda)$ for $\lambda$ a partition of $r$ form a complete set of simple modules for a semisimple symmetric group algebra $FS_r$. In finite characteristic $p$ the Specht modules corresponding to $p$-regular partitions have simple heads which form a complete set of simple modules of $FS_r$. For partition algebras, the analogues of the Specht modules are given as follows:

\begin{theorem}[Corollary 2.3 in \cite{DW2000}]
\label{thm classification of simples}
Let $\lambda$ be a partition of some natural number $l$ with $l \leq r$ and denote by $S(\lambda)$ the irreducible Specht module for the symmetric group $S_l$. Set $$\Delta_r(\lambda)=V_l \otimes_{S_l} S(\lambda) \textrm{ \ and \ } \Delta_{r+1/2}(\lambda)=V_{l+1/2} \otimes_{S_l} S(\lambda).$$
If $P_{d}(\delta)$ is semisimple, then the set $\{ \Delta_d(\lambda) \ | \ \lambda \in \Lambda^+(d)\} $  are precisely the cell modules and are, in particular,  a complete set of simple $P_d(\delta)$-modules. If $\delta=0$, we exclude $\lambda=\emptyset$.

If $P_d(\delta)$ is not semisimple, denote by $L_d(\lambda)$ the head of $\Delta_d(\lambda)$. If $\lambda$ is $p$-regular, then $L_d(\lambda)$ is simple and the set $\{ L_d(\lambda) \ | \ \lambda \in \Lambda^+(d), \textrm{ $\lambda$ $p$-regular}\} $ is a complete set of simple modules for $P_d(\delta)$. Again, if $\delta = 0$, we exclude $\lambda=\emptyset$.
\end{theorem}

One aim of this paper is to understand the composition factors of the cell modules $\Delta(\lambda)$. We denote by $d_{\lambda\mu}=[\Delta(\lambda):L(\mu)]$ the composition multiplicity of the simple module $L(\mu)$ in  the cell module $\Delta(\lambda)$. The decomposition matrix $D_{d}$ is the matrix with rows labelled by $\Lambda^+(d)$ and columns labelled by the $p$-regular partitions in $\Lambda^+(d)$   and entry $(\lambda,\mu)$ equal to $d_{\lambda \mu}$. The following theorem follows from \cite{4authors}:

\begin{theorem}[Proposition 6.1 and Corollary 6.2 in \cite{4authors}]
\label{thm form of dec matrix} 
There is an ordering of  cell and simple modules  such that $$  D_{ P_d(\delta) }=\left( \begin{array}{cc}
D_{FS_r} & 0  \\ 
 \ast & D_{P_{d-1}(\delta)}
\end{array}\right)$$ where $D_{FS_r}$ denotes the decomposition matrix of the symmetric group and where we leave out the row labelled by $\emptyset$ if $\delta=0$.
In particular,  
$D_{P_d(\delta)}$ is lower triangular with all diagonal entries equal to $1$ and contains the decomposition matrices of the symmetric group algebras $FS_l$ for all $l \leq r$.
\end{theorem}

Over $\C$, the following theorem first appeared in \cite{Martin96} and was also proved in \cite{DW2000}.
\begin{theorem}
\label{thm differ by one box then common CF}
Suppose the characteristic $p$ of the field is arbitrary. Let $\mu$ be a partition of $r$ and $\lambda$ a partition of $r-1$ such that $\lambda=\mu-\epsilon$ for some removable box $\epsilon$ of $\mu$.   Then $[\Delta(\lambda):L(\mu)] \neq 0$ over $P_d(\delta)$ if and only if $c(\epsilon)\equiv \delta-r+1 \pmod{p}$.
\end{theorem}

\begin{proof}
In characteristic $0$ this is well-known, see for example Theorems 5.1 and 7.1 in \cite{DW2000} . The proof of Theorem 5.1 is also valid over an arbitrary field which shows one direction. Conversely, if the content of $\epsilon$ is congruent to $\delta-r+1$ modulo $p$, then it is equal to $\delta-r+1+kp$ for some $k \in \Z$. Thus for the partition algebra $P_{r}(\delta+kp)$, we have  $[\Delta(\lambda):L(\mu)] \neq 0$. Reducing modulo $p$, we hence also get $[\Delta(\lambda):L(\mu)] \neq 0$ for the  partition algebra $P_r(\delta)$ over a field of characteristic $p$. 
\end{proof}

\begin{rk}
It is possible to prove this result using Young's orthogonal form as introduced in \cite{Enyang13} only. One can start in the generic case and show that for this special choice of $\lambda$ and $\mu$, only few of the coefficients in Young's orthogonal form have denominators divisible by $x-\delta$ and eliminate these by summing over basis vectors corresponding to tableau of the same weight. It is then possible to reduce modulo $x-\delta$ which yields the required result by considering the (explicitly given) action of the generators of the partition algebra. However, this approach is quite technical and long, so we have decided not to include it in this paper. But it is worth mentioning that, in principle, the approach of studying blocks and decomposition numbers by analysing only the action of JM elements is possible.
\end{rk}

\subsection{Induction, restriction and globalization}
Since we can embed $P_{d-1/2}$ in $P_{d}(\delta)$ (for $d \in \{r,r+1/2\}$), we can define the usual induction and restriction functors corresponding to this inclusion which will be denoted by \begin{align*}
\textrm{ind}_d: P_d(\delta) \text{-mod} \to P_{d+1/2}(\delta) \text{-mod}  \hspace*{0.8cm} \textrm{res}_{d+1/2}: P_{d+1/2}(\delta) \text{-mod} \to P_{d}(\delta)\text{-mod}.
\end{align*} Notice that there are various ways to embed a smaller in a larger partition algebra, which will however not affect the functors up to isomorphism.

There is also a second way to relate modules over different partition algebras. For this we use the idempotent $e=e_1$ from Definition \ref{defn of the stratification data}. Notice that $eP_{d}(\delta)e$ is isomorphic to $P_{d-1}(\delta)$ and this leads to the well-known (see e.g. Section 6 in \cite{Green}) globalization functor \begin{align*}
\textrm{G}_d: P_d(\delta)\text{-mod}  & \to P_{d+1}(\delta)\text{-mod}  \\ M & \mapsto P_{d+1}(\delta)e \otimes_{eP_{d+1}(\delta)e} M.
\end{align*}

Notice that because $e_1e_l=e_l$, the globalization and layer induction functor are related via  $G G_{l,d}(M)=G_{l,d+1} M$ for any $FS_l$-module $M$. In particular, we obtain:

\begin{prop}
Let $\lambda \in \Lambda^+(d)$. Then $G\Delta_d(\lambda)=\Delta_{d+1}(\lambda)$.
\end{prop}

Induction and restriction of the cell modules of the partition algebra were first studied by Martin \cite{Martin96} \cite{Martin00} and then later in more detail by Doran and Wales \cite{DW2000}. However,   Doran and Wales only studied induction and restriction between the full partition algebras and omitted the half steps and Martin only studied the semisimple case.  The following theorem provides induction and restriction rules which are valid over an arbitrary field. In the theorem, we denote by $M\uparrow^{S_{l+1}}$ and $M \downarrow_{S_{l-1}}$ the induction and restriction of an $FS_l$-module $M$ to $FS_{l+1}$ and $FS_{l-1}$, respectively.
\begin{theorem}
\label{thm restriction of cells}
Suppose the characteristic of the field is arbitrary, $l \leq n$ is a natural number and let $M$ be an arbitrary $S_l$-module . Then the following sequences are exact:\\
$(a)$ \ $\displaystyle 0 \longrightarrow V_{l-1/2} \otimes M\downarrow_{S_{l-1}} \longrightarrow \res_r ( V_l \otimes M) \longrightarrow V_{l+1/2} \otimes M \longrightarrow 0. $\\
$(b)$ \ $\displaystyle 0 \longrightarrow V_{l} \otimes M\longrightarrow \res_{r+1/2} (  V_{l+1/2} \otimes M) \longrightarrow V_{l+1} \otimes M \uparrow^{S_{r+1}} \longrightarrow 0. $ \\
$(c)$ \ $\displaystyle 0 \longrightarrow V_{l-1/2} \otimes M\downarrow_{S_{l-1}} \longrightarrow \ind_r ( V_l \otimes M) \longrightarrow V_{l+1/2} \otimes M \longrightarrow 0. $\\
$(d)$ \ $\displaystyle 0 \longrightarrow V_{l} \otimes M\longrightarrow \ind_{r-1/2} (  V_{l+1/2} \otimes M) \longrightarrow V_{l+1} \otimes M \uparrow^{S_{r+1}} \longrightarrow 0. $\\
In particular, the following sequences are exact, provided that  the characteristic $p$ of the field is $0$ or larger than $r$:\\
$(e)$ \ $\displaystyle 0 \longrightarrow  \bigoplus_{\epsilon \in \textrm{rem}(\lambda)} \Delta(\lambda-\epsilon)  \longrightarrow \res_r  \Delta(\lambda) \longrightarrow \Delta(\lambda) \longrightarrow 0. $\\
$(f)$ \ $\displaystyle 0 \longrightarrow\Delta(\lambda)\longrightarrow \res_{r+1/2} \Delta(\lambda) \longrightarrow \bigoplus_{\epsilon \in \textrm{add}(\lambda)} \Delta(\lambda+\epsilon) \longrightarrow 0. $ \\
$(g)$ \ $\displaystyle 0 \longrightarrow  \bigoplus_{\epsilon \in \textrm{rem}(\lambda)} \Delta(\lambda-\epsilon)\longrightarrow \ind_r \Delta(\lambda) \longrightarrow \Delta(\lambda) \longrightarrow 0. $\\
$(h)$ \ $ \displaystyle 0 \longrightarrow\Delta(\lambda)\longrightarrow \ind_{r-1/2} \Delta(\lambda) \longrightarrow \bigoplus_{\epsilon \in \textrm{add}(\lambda)} \Delta(\lambda+\epsilon) \longrightarrow 0. $\\
\end{theorem}

\begin{proof}
This was essentially already proved in detail in Theorem 11.1 of \cite{DW2000} except that they did not take the half-steps into account. We will therefore omit some details in this proof. For the proofs of the first two exact sequences, we use restriction functors based on the embedding of $P_r(\delta)$ in $P_{r+1/2}(\delta)$ by adding a vertical edge on the right of each diagram. $P_{r+1/2}(\delta)$ is embedded in $P_{r+1}(\delta)$ in a natural way by definition. 

 We start with the second exact sequence.  Define a map $\theta:V_{l+1/2} \otimes_{S_l} M \to V_{l+1} \otimes_{S_{l+1}} S_{l+1} \otimes_{S_l} M$ by $$d \otimes m \mapsto A_{r+1}\cdot d \cdot (r-l,r+1) \otimes 1 \otimes m$$ for any $m \in M$ and any diagram $d \in V_{l+1/2}$ and extending linearly. This map can also be described graphically: $ A_{r+1}\cdot d \cdot (r-l,r+1)$ can be constructed from $d$ by connecting the block of $d$ which contains $r+1$ to $(r-l)'$ instead of $(r+1)'$ and then replacing column $r+1$ of $d$ by two singleton dots deleting all edges it is connected to. Because $A_{r+1}$ commutes with every element of $P_r(\delta)$ it is easy to see that $\theta$ is a homomorphism of left $P_r(\delta)$-modules. 

Now the codomain of the map is precisely $V_{l+1} \otimes M \uparrow^{S_{l+1}}$ so that we will first show that the map is surjective. Let $d \otimes (i,r+1) \otimes m$ be an element of  $V_{l+1} \otimes_{S_{l+1}} S_{l+1} \otimes_{S_l} M$ for some diagram $d \in V_{l+1}$, some $m \in M$ and $i \in \{1,\ldots,r+1\}$. Notice that such elements span $V_{l+1} \otimes_{S_{l+1}} S_{l+1} \otimes_{S_l} M$ because transpositions of the form $(i,l+1)$ are a complete set of representatives for the cosets of $S_l$ in $S_{l+1}$. But by the graphical description of $\theta$ mentioned above, we can thus construct a diagram $d'$ such that  $A_{r+1}\cdot d' \cdot (r-l,r+1)$ is equal to $d\cdot (i,r+1)$: Start with $d\cdot (i,r+1)$ and add a column on the right of the diagram, multiply from the right by $(r-l,r+1)$ (i.e. connect the block which was connected to $(r-l)'$ with $(r+1)'$) and connect $r+1$ to $(r+1)'$. Thus $\theta(d' \otimes m)=d \otimes (i,r+1) \otimes m$ and we have proved surjectivity. 

Notice that for any diagram $d$ which has a block of the form $(r+1)(r+1)'$ (i.e. $(r+1)$ and $(r+1)'$ are connected to each other but to nothing else) will satisfy $\theta(d \otimes m)=0$, since this means that $(r-l)$ will be a singleton block of $A_{r+1}\cdot d \cdot (r-l,r+1)$ and hence has less than $l+1$ propagating lines. Thus, if we embed $V_{l,r}$ in $V_{l+1,r+1}$ by adding a vertical edge at the end, then we see that $V_{l,r} \otimes M \subseteq \ker \theta$. 

The dimension of $V_{l+1} \otimes S_{l+1} \otimes_{S_l} S_l$ is $(l+1) \displaystyle \sum_{k=l+1}^r  S(r,k) \left(\begin{array}{c}k \\ l+1
\end{array} \right)$ 
where $S(r,k)$ is the number of (set) partitions of $r$ into at least $k$ parts: There are $S(r,k)$ ways to choose the partition and then we have to choose $l$ out of the $k$ parts which are connected to the bottom row. On the other hand the dimension of 
$V_{l+1/2,r+1/2}$ is equal to the sum of the dimension of $V_{l,r}$ (all diagrams which have a vertical edge at the end) and $\sum_{k=l}^r S(r,k)\left(\begin{array}{c}k \\ l
\end{array} \right)(k-l)$: The number of partitions of $r$ into $k$ parts of which $l$ are connected to the bottom row times the $k-l$ choices we can connect $r+1$ to. An easy calculation shows, that the dimension of the kernel of $\theta$ must be equal to the dimension of $V_{l,r} \otimes M$ and the result follows.

Now consider the first exact sequence. This time we want to construct a map $$\theta: V_{l} \otimes_{S_l} M \to V_{l+1/2}  \otimes_{S_l} M.$$ In order to do this we will first construct a map $f:  V_{l} \cdot (r,r-l) \to V_{l+1/2} $ by defining $f(d)$ to be the same diagram as the diagram $d \in  V_{l} \cdot (r,r-l)$ except that we join the blocks containing $r$ and $r'$ and extending this linearly to the whole of $ V_{l} \cdot (r,r-l)$. Notice that every element in  $V_{l} $ has a singleton dot at position $(r-l)'$ and therefore $d$ has a singleton dot at position $r'$. We then define $$\theta(d \otimes m)=f(d\cdot (r,r-l)) \otimes m$$ for $d \in V_{l}$ and $m \in M$.

To show that $\theta$ is a $P_{r-1/2}(\delta)$-module homomorphism, it suffices to show that $f$ is. But if $f(d)=d'$ for some diagram $d \in  V_{l} \cdot (r,r-l)$ and $d' \in V_{l+1/2}$, then multiplication by $A_{r}$ exactly inverts the operation of $f$ since in the diagram $d$ the dots $r$ and $r'$ are not connected. So $d=d'A_{r}$ and $ad=ad'A_r$ for any diagram $a \in P_{r-1/2}(\delta)$. But $ad' \in V_{l+1/2}$ so that $f(ad)=ad'=af(d)$ which shows that $f$ and therefore $\theta$ is a $P_{r-1/2}(\delta)$-module homomorphism.

Notice that $\theta$ again has a graphical description: The second tensor component stays unchanged while in the first we first swap $r'$ and $(r-l)'$ in the bottom row and then connect $r$ and $r'$. This can simply be undone by removing the connection between $r$ and $r'$ and then swapping $r'$ and $(r-l)'$ so that surjectivity follows similarly as in the first case.

We have to study the kernel of this map. If $r$ and $r'$ are connected in $d$, then $f(d\cdot(r,r-l))$ will have the same number of propagating lines as $d$ and thus be zero in $V_{l+1/2}$. Thus, $V_{l-1/2} \otimes_{S_l} M$ is contained in $\ker \theta$ and because $r$ and $r'$ are always connected in $V_{l-1/2}$, $V_{l-1/2} \otimes_{S_l} M$ is isomorphic to $V_{l-1/2} \otimes_{S_{l-1}} M \downarrow_{S_{l-1}}$ as $P_{r-1/2}(\delta)$-module. The statement now follows by a dimension argument, because \begin{align*}
 \dim V_{l,r}=\sum_{k=l}^r S(r,k) \left(\begin{array}{c}
k \\ 
l
\end{array}\right), \dim V_{l-1/2,r-1/2} & =\sum_{k=l}^r S(r,k) \left(\begin{array}{c}
k-1 \\ 
l-1
\end{array}\right), \\ \displaystyle \dim V_{l+1/2,r-1/2} & =\sum_{k=l+1}^r S(r,k) \left(\begin{array}{c}
k-1 \\ 
l
\end{array}\right).
\end{align*}

The exact sequences $(c)$ and $(d)$ can be proved by noticing that inducing is the same as globalizing and then restricting, that is, $\ind_d M = \res_{d+1} G M$ for any $P_d(\delta)$-module $M$, see Proposition 7 in \cite{Martin96} for the integer degree case.
Thus, we have to show that the vector spaces $ P_{r+1}(\delta)e_1 \otimes_{P_{r}(\delta)} M$ (i.e. we first apply globalization and then restrict) and $P_{r+1/2}(\delta) 
 \otimes_{P_{r}(\delta)} M$ (ordinary induction)  are isomorphic as $P_{r+1/2}(\delta)$-modules. Furthermore, we have to show that $ P_{r+1/2}(\delta)e_1 \otimes_{P_{r-1/2}(\delta)} M$ and $P_{r}(\delta) 
 \otimes_{P_{r-1/2}(\delta)} M$ are isomorphic as $P_r(\delta)$-modules.

We will see that we have already essentially solved this problem by rewriting it. First notice that $P_{r+1}(\delta)e_1 \cong P_{r+1}(\delta)A_1$  as $P_{r+1/2}(\delta)$-$P_r(\delta)$-bimodules where the right action of $P_{r}(\delta)$ on 
$P_{r+1}(\delta)A_1$ is induced by embedding $P_{r}(\delta)$ in $P_{r+1}(\delta)$ by adding a vertical edge \textit{on the left} of all diagrams. The isomorphism is given by right multiplication by $A_1$ which obviously commutes with the action of $P_{r+1/2}(\delta)$ on the left and also with $P_{r}(\delta)$ on the right by the choice of embedding of $P_r(\delta)$ in $P_{r+1}(\delta)$. The inverse is given by connecting $1'$ and $2'$ in each diagram which works since $1'$ and $2'$ are connected in each diagram of $P_{r+1}(\delta)e_1$. Using exactly the same reasoning shows that $P_{r+1/2}(\delta)e_1$ and $P_{r+1/2}(\delta)A_1$ are isomorphic as $P_{r}(\delta)$-$P_{r-1/2}(\delta)$-bimodules.

Thus, we have to show $ P_{r+1}(\delta)A_1 \otimes_{P_{r}(\delta)} M \cong P_{r+1/2}(\delta) 
 \otimes_{P_{r}(\delta)} M$. But now the map is induced by the map $f$ given above  composed with multiplication by the permutation $(1,r+1)$ from the right, the proof being almost identical as before. Similarly, $ P_{r+1/2}(\delta)A_1 \otimes_{P_{r-1/2}(\delta)} M \cong P_{r}(\delta) 
 \otimes_{P_{r-1/2}(\delta)} M$ where the isomorphism is induced by the map $d \in P_{r+1/2}(\delta)A_1 \mapsto A_{r+1} \cdot d \cdot (1,r+1)$ which has a similar graphical description as the map for the first exact sequence. The proof that it is an isomorphism is almost identical.

The exact sequences $(e)$-$(h)$ are a consequence of the branching rules for symmetric groups, see e.g. Theorem 9.2 in \cite{James}.
\end{proof}

\subsection{Jucys-Murphy elements and canonical basis}
As shown in Theorem 3.37 of \cite{HalversonRam}, the simple modules for the partition algebra have a canonical basis which is analogous to Gelfand-Zetlin bases, provided we are in the semisimple case.
 The construction relies  on the multiplicity freeness of the restriction from $P_d(\delta)$ to $P_{d-1/2}(\delta)$, which was established in Theorem \ref{thm restriction of cells}. Thus, we can restrict a given simple module which will yield a unique direct sum decomposition into simple modules and by iterating the procedure for the summands, we eventually end up with unique direct sum decomposition into $1$-dimensional subspaces. We can label the subspaces as follows:
To each $\lambda \in \Lambda^+(d)$, we can associate a so called Bratelli diagram. The vertices at level $i \in 1/2 \N$ are given by the direct summands of the restriction of $\Delta(\lambda)$ to $P_{d-i}(\delta)$ and there is an edge connecting a vertex $v$ in row $i$ with a vertex $v'$ in row $i+1/2$ if the module corresponding to $v'$ is a direct summand of the restriction of the module corresponding to $v$ to $P_{r-i-1/2}(\delta)$. The paths from top to bottom are called tableaux and are precisely the labels of the $1$-dimensional subspaces above:

\begin{defn}
 A  tableau $t$ of a partition $\lambda$  is a sequence of partitions \newline $(t^{1/2}=\emptyset,t^1,t^{3/2},\ldots,t^{d}=\lambda)$ such that for all $k=1/2,1,3/2,\ldots,d$ the partitions $t^{k}$ and $t^{k-1/2}$ are either equal or differ by one box, where the box has to be added to $t^{k-1/2}$ if $k \in \N$ and removed from $t^{k-1/2}$ if $k \notin \N$. We will sometimes refer to the transition from $t^{k-1/2}$ to $t^{k}$ by adding or removing a box as the $k$th step of $t$. 
 The set of tableaux of a given partition $\lambda \in \Lambda^+(d)$  is denoted by $\textrm{Tab}(d,\lambda)$  or just $\rm{Tab}(\lambda)$ if  $d$ is clear from the context.
\end{defn}

\begin{ex}
\label{ex updown tab} Examples of tableaux for $d=4+1/2$ are $$t=\left( {\Yvcentermath1 \begin{array}{c}
 \emptyset,\yng(1),\yng(1), \yng(2),\yng(2), \yng(2,1), \yng(2,1), \yng(3,1),\yng(3,1)
 \end{array}} \right),$$ $$u=\left( {\Yvcentermath1 \begin{array}{c}
\emptyset,\yng(1),\yng(1), \yng(2),\yng(2), \yng(2,1), \yng(2,1), \yng(2,1),\yng(1,1)
 \end{array}} \right).$$ 
 \end{ex}

Tableaux play a similar role for partition algebras as standard tableaux for the symmetric group. In particular, if we pick a non-zero element from each of the $1$-dimensional subspaces above, then we obtain a basis of the cell modules which is labelled by tableaux in a natural way.

In the symmetric group case, canonical basis vectors have the important property that they are simultaneous eigenvectors for the action of Jucys-Murphy (JM) elements. Analogues of these JM elements were also defined for partition algebras by Halverson and Ram \cite{HalversonRam} where they are just called Murphy elements. The exact definition of the Jucys-Murphy elements will not be important in this paper and we refer the reader to Section $3$ of \cite{HalversonRam}. We will only need the following theorem:

\begin{theorem}[Theorem 3.37 in \cite{HalversonRam}]
\label{thm eigenvals of JM action}
Let $d \in \{r,r+1/2 \}$ and suppose that  $L_{1/2},L_1,L_{3/2},\ldots L_d \in P_d(\delta)$ are the Jucys-Murphy elements as defined in \cite{HalversonRam}. Then:
\begin{enumerate}[(i)]
\item The elements $L_{1/2},L_1,L_{3/2},\ldots L_d$ commute with each other.
\item Assume $P_d(\delta)$ is semisimple and let $\{ v_t\}_{t \in \UD(\lambda)}$ be the canonical basis of $\Delta(\lambda)$. Then each Jucys-Murphy element $L_k$ acts by a scalar $L_k(t)$ on $v_t$ for $t \in \UD(\lambda)$ with:
$$L_k(t)= 
\begin{cases}
c(\epsilon) & \textrm{if } t^{k}=t^{k-1/2}+\epsilon, k \in \N \\
\delta - |t^k| & \textrm{if } t^{k}= t^{k-1/2}, k \in \N  \\
\delta-c(\epsilon) & \textrm{if } t^{k}=t^{k-1/2}-\epsilon, k \notin \N \\
|t^k| & \textrm{if } t^{k}= t^{k-1/2}                                                                                                                                    
, k \notin \N  \end{cases}.$$
\end{enumerate}
\end{theorem}

\begin{defn}
Let $\lambda \in \Lambda^+(d)$. The weight $\wt(t)$ of a $\lambda$-tableau $t$ is the tuple  $(L_{1/2}(t),L_{1}(t),L_{3/2}(t),\ldots L_d(t))$. We say that $\lambda,\mu \in \Lambda^+(d)$ have a common JM weight if there is a $\mu$-tableau $u$ and  a $\lambda$-tableau $t$ with $\wt(u)=\wt(t)$.
\end{defn}

\begin{ex}
We represent a tableau together with its weight by a coloured graph as follows:
 $$t: {\Yvcentermath1 \begin{array}{c}
 \emptyset \stackrel{0}{\overline{ \ \ \ \ }} \yng(1)\stackrel{1}{\overline{ \ \ \ \ }}\yng(1)\stackrel{1}{\overline{ \ \ \ \ }} \yng(2) \stackrel{2}{\overline{ \ \ \ \ }} \yng(2) \stackrel{-1}{\overline{ \ \ \ \ }} \yng(2,1) \stackrel{3}{\overline{ \ \ \ \ }} \yng(2,1) \stackrel{2}{\overline{ \ \ \ \ }} \yng(3,1) \stackrel{4}{\overline{ \ \ \ \ }} \yng(3,1)
 \end{array}},$$ $$u: {\Yvcentermath1 \begin{array}{c}
\emptyset\stackrel{0}{\overline{ \ \ \ \ }} \yng(1) \stackrel{1}{\overline{ \ \ \ \ }} \yng(1) \stackrel{1}{\overline{ \ \ \ \ }}  \yng(2) \stackrel{2}{\overline{ \ \ \ \ }} \yng(2) \stackrel{-1}{\overline{ \ \ \ \ }} \yng(2,1) \stackrel{3}{\overline{ \ \ \ \ }} \yng(2,1) \stackrel{\delta-3}{\overline{ \ \ \ \ }} \yng(2,1) \stackrel{\delta-1}{\overline{ \ \ \ \ }} \yng(1,1)
 \end{array}}. $$ Notice that for $\delta=5$ the tableau $u$ and $t$ have the same weight. Thus, $(3,1)$ and $(1,1)$ have a common JM weight. One can prove that $[\Delta(1,1):L(3,1)] \neq 0$ and we will show in Theorem \ref{thm determination of dec nos} that this is not a coincidence.
\end{ex}

For the definition of projection functors onto JM weight spaces, we also need the following theorem:

\begin{theorem}[Theorem 3.10 in \cite{Enyang13}] \label{thm sum of JMs central}
The element $\displaystyle z= \sum_{k \in 1/2\N, k \leq d} L_k \in P_d(\delta) $ is central in $P_d(\delta)$.
\end{theorem}

\section{Combinatorial setup}
In this section, we will introduce diagrams which are similar to the arrow diagrams introduced by Brundan and Stroppel in \cite{BrundanStroppel}. These diagrams also appear in the work of Cox and De Visscher \cite{CoxDeVisscher} where the calculus is used to determine the decomposition numbers of Brauer algebras over $\C$. The author has shown how this diagrammatic calculus can be naturally derived from the combinatorics of JM elements in the Brauer algebra case \cite{JMpaper}. We will adopt a similar approach here. The arrow diagrams in question are defined as follows:

\begin{defn}
Given a partition $\lambda=(\lambda_1,\lambda_2,\ldots) \in \Lambda^+(d)$, (where we allow zero entries) we define its diagram or arrow diagram $d(\lambda)$ as follows:

If the characteristic of the field is $0$, then the arrow diagram consists of a doubly infinite line with positions above and below the line labelled as follows:
\vspace{0.5cm}
$$
\begin{picture}(0,0)(100,5) 
\multiput(-14,5.5)(35,0){7}{$\scriptstyle\bullet$}
\multiput(22.5,8)(35,0){6}{\line(-1,0){35} }
\put(212.5,8){\line(-1,0){15} }
\put(220,4.5){$\cdots$}
\put(-40,4.5){$\cdots$}

\put(-11.5,20){\makebox(	0	,0){$\scriptstyle  3$}}
\put(	23.5	,20){\makebox( 0 ,0){$\scriptstyle 2$}}
\put(	58.5	,20){\makebox( 0 ,0){$\scriptstyle 1$}}
\put(	93.5	,20){\makebox( 0 ,0){$\scriptstyle 0 $}}
\put(	128.5	,20){\makebox( 0 ,0){$\scriptstyle -1$}}
\put(	163.5	,20){\makebox( 0 ,0){$\scriptstyle -2$}}
\put(	198.5	,20){\makebox( 0 ,0){$\cdots$}}

\put(-11.5,-5){\makebox(	0	,0){$\scriptstyle \delta-3 $}}
\put(	23.5	,-5){\makebox( 0 ,0){$\scriptstyle \delta-2 $}}
\put(	58.5	,-5){\makebox( 0 ,0){$\scriptstyle \delta-1 $}}
\put(	93.5	,-5){\makebox( 0 ,0){$\scriptstyle \delta $}}
\put(	128.5	,-5){\makebox( 0 ,0){$\scriptstyle \delta+1 $}}
\put(	163.5	,-5){\makebox( 0 ,0){$\scriptstyle  \delta+2$}}
\put(	198.5	,-5){\makebox( 0 ,0){$\cdots$}}
\end{picture}$$   
\vspace{0.1cm}
\newline Thus, opposite labels always sum to $\delta$. For each part $\lambda_i$ of $\lambda$, we draw a $\bigvee$ above the line at position $\lambda_i-i$ (that is, at the content of the last box of each row). Furthermore, we draw a single $\bigwedge$ below the line at position $|\lambda|$.

If the characteristic $p$ of the field  is non-zero, then the arrow diagram is very similar to the usual abacus defined for symmetric groups, see for example \cite{JamesKerber}. The arrow diagram/abacus is defined as follows: It consists of an infinite number of lines with $p$ positions each and labelled by $k \in \Z$ called the runners of the diagram. There is one distinguished runner corresponding to $k=0$ which has labels above and below the $p$ positions, but where the labels below the runner are considered modulo $p$.
\vspace{0.5cm}
$$
\begin{picture}(0,0)(100,5) 
\multiput(-14,5.5)(35,0){6}{$\scriptstyle\bullet$}
\multiput(22.5,8)(35,0){5}{\line(-1,0){35} }

\put(-11.5,20){\makebox(	0	,0){$\scriptstyle  p-1$}}
\put(	23.5	,20){\makebox( 0 ,0){$\scriptstyle p-2$}}
\put(	58.5	,20){\makebox( 0 ,0){$\scriptstyle \cdots$}}
\put(	93.5	,20){\makebox( 0 ,0){$\scriptstyle 2 $}}
\put(	128.5	,20){\makebox( 0 ,0){$\scriptstyle 1$}}
\put(	163.5	,20){\makebox( 0 ,0){$\scriptstyle 0$}}

\put(-11.5,-5){\makebox(	0	,0){$\scriptstyle \delta+1 $}}
\put(	23.5	,-5){\makebox( 0 ,0){$\scriptstyle \delta+2 $}}
\put(	58.5	,-5){\makebox( 0 ,0){$\scriptstyle \cdots $}}
\put(	93.5	,-5){\makebox( 0 ,0){$\scriptstyle \delta-2 $}}
\put(	128.5	,-5){\makebox( 0 ,0){$\scriptstyle \delta-1 $}}
\put(	163.5	,-5){\makebox( 0 ,0){$\scriptstyle  \delta$}}

\end{picture}$$   
\vspace{0.1cm}\\
The runners corresponding to $k \neq 0$ are labelled as follows:
\vspace{0.5cm}
$$
\begin{picture}(0,0)(100,5) 
\multiput(-14,5.5)(35,0){6}{$\scriptstyle\bullet$}
\multiput(22.5,8)(35,0){5}{\line(-1,0){35} }

\put(-11.5,20){\makebox(	0	,0){$\scriptstyle  kp+p-1$}}
\put(	23.5	,20){\makebox( 0 ,0){$\scriptstyle kp+p-2$}}
\put(	58.5	,20){\makebox( 0 ,0){$\scriptstyle \cdots$}}
\put(	93.5	,20){\makebox( 0 ,0){$\scriptstyle kp+2 $}}
\put(	128.5	,20){\makebox( 0 ,0){$\scriptstyle kp+1$}}
\put(	163.5	,20){\makebox( 0 ,0){$\scriptstyle kp$}}

\end{picture}$$  
The runners are vertically stacked on top of each other in such a way that in each column of the resulting diagram the labels are congruent modulo $p$. We refer to the column containing labels congruent to $i$ modulo $p$ as column $i$.\\ 
As in the characteristic $0$ case, we draw a $\bigvee$ at position $\lambda_i-i$ for each part $\lambda_i$. Furthermore, we draw a $\bigwedge$ at position $|\lambda|$ but now considered modulo $p$.
\end{defn}

\begin{rk}
Notice that $\lambda$ and $d(\lambda)$ determine each other and we will often identify them. Furthermore, adding a box to $\lambda$ in some row corresponds to move the arrow $\bigvee$ corresponding to that row up by $1$ in the arrow diagram, that is to a position with a higher label. Also, a box is addable if and only if the next higher position is empty in the arrow diagram. Similar statements hold for removable boxes. We will see in Figure~\ref{figure directions to add arrows} that arrow diagrams are additionally particularly well behaved with respect to the action of the JM elements.
\end{rk}
\begin{ex}
Both definitions are best understood by an example. We consider the partition $\lambda=(3,2,2,2,1,0,0,\ldots)$ with $\delta=8$. The contents of the last boxes in each row are $(3,1,0,-1,-3,-5,-6,\ldots)$ and $|\lambda|=10$. Therefore, if $p=0$ we get the following arrow diagram $d(\lambda)$:
\vspace{0.5cm}
$$
\begin{picture}(0,0)(160,5) 
\multiput(-14,5.5)(35,0){11}{$\scriptstyle\bullet$}
\multiput(22.5,8)(35,0){10}{\line(-1,0){35} }

\put(-25.5,7.5){\makebox(	0	,0){$  \cdots$}}
\put(355,7.5){\makebox(	0	,0){$  \cdots$}}

\put(	23.5	,13){\makebox( 0 ,0){$\scriptstyle \bigvee$}}
\put(	93.5	,13){\makebox( 0 ,0){$\scriptstyle \bigvee $}}
\put(	128.5	,13){\makebox( 0 ,0){$\scriptstyle \bigvee$}}
\put(163.5,13){\makebox(	0	,0){$\scriptstyle  \bigvee$}}
\put(	233.5	,13){\makebox( 0 ,0){$\scriptstyle \bigvee$}}
\put(	303.5	,13){\makebox( 0 ,0){$\scriptstyle \bigvee$}}
\put(	338.5	,13){\makebox( 0 ,0){$\scriptstyle \bigvee$}}

\put(	163.5	,3){\makebox( 0 ,0){$\scriptstyle \bigwedge$}}

\put(-11.5,25){\makebox(	0	,0){$\scriptstyle  3$}}
\put(	23.5	,25){\makebox( 0 ,0){$\scriptstyle 2$}}
\put(	58.5	,25){\makebox( 0 ,0){$\scriptstyle 1$}}
\put(	93.5	,25){\makebox( 0 ,0){$\scriptstyle 0 $}}
\put(	128.5	,25){\makebox( 0 ,0){$\scriptstyle -1$}}
\put(163.5,25){\makebox(	0	,0){$\scriptstyle  -2$}}
\put(	198.5	,25){\makebox( 0 ,0){$\scriptstyle -3$}}
\put(	233.5	,25){\makebox( 0 ,0){$\scriptstyle -4$}}
\put(	268.5	,25){\makebox( 0 ,0){$\scriptstyle -5 $}}
\put(	303.5	,25){\makebox( 0 ,0){$\scriptstyle -6$}}
\put(	338.5	,25){\makebox( 0 ,0){$\scriptstyle -7$}}

\put(-11.5,-8){\makebox(	0	,0){$\scriptstyle  5$}}
\put(	23.5	,-8){\makebox( 0 ,0){$\scriptstyle 6$}}
\put(	58.5	,-8){\makebox( 0 ,0){$\scriptstyle 7$}}
\put(	93.5	,-8){\makebox( 0 ,0){$\scriptstyle 8 $}}
\put(	128.5	,-8){\makebox( 0 ,0){$\scriptstyle 9$}}
\put(163.5,-8){\makebox(	0	,0){$\scriptstyle  10$}}
\put(	198.5	,-8){\makebox( 0 ,0){$\scriptstyle 11$}}
\put(	233.5	,-8){\makebox( 0 ,0){$\scriptstyle 12$}}
\put(	268.5	,-8){\makebox( 0 ,0){$\scriptstyle 13 $}}
\put(	303.5	,-8){\makebox( 0 ,0){$\scriptstyle 14$}}
\put(	338.5	,-8){\makebox( 0 ,0){$\scriptstyle 15$}}

\end{picture}$$   
\vspace{0.1cm}\\
If $p=5$, and $\lambda=(3,2,2,2,1,0,0,\ldots)$ and $\delta=8$ are as before, then the arrow diagram would look like:
\vspace{2cm}
$$
\begin{picture}(0,0)(70,5) 
\multiput(-14,5.5)(35,0){5}{$\scriptstyle\bullet$}
\multiput(22.5,8)(35,0){4}{\line(-1,0){35} }

\put(-11.5,14){\makebox(	0	,0){$\bigvee$}}
\put(23.5,14){\makebox(	0	,0){$\bigvee$}}
\put(58.5,14){\makebox(	0	,0){$\bigvee$}}
\put(93.5,14){\makebox(	0	,0){$\bigvee$}}
\put(128.5,14){\makebox(	0	,0){$\bigvee$}}

\put(-11.5,25){\makebox(	0	,0){$\scriptstyle  -6$}}
\put(	23.5	,25){\makebox( 0 ,0){$\scriptstyle -7$}}
\put(	58.5	,25){\makebox( 0 ,0){$\scriptstyle -8$}}
\put(	93.5	,25){\makebox( 0 ,0){$\scriptstyle -9 $}}
\put(	128.5	,25){\makebox( 0 ,0){$\scriptstyle -10$}}

\end{picture}$$   
\vspace{0.5cm}
$$
\begin{picture}(0,0)(70,5) 
\multiput(-14,5.5)(35,0){5}{$\scriptstyle\bullet$}
\multiput(22.5,8)(35,0){4}{\line(-1,0){35} }

\put(-11.5,14){\makebox(	0	,0){$\bigvee$}}
\put(23.5,14){\makebox(	0	,0){$\bigvee$}}
\put(93.5,14){\makebox(	0	,0){$\bigvee$}}

\put(-11.5,25){\makebox(	0	,0){$\scriptstyle  -1$}}
\put(	23.5	,25){\makebox( 0 ,0){$\scriptstyle -2$}}
\put(	58.5	,25){\makebox( 0 ,0){$\scriptstyle -3$}}
\put(	93.5	,25){\makebox( 0 ,0){$\scriptstyle -4 $}}
\put(	128.5	,25){\makebox( 0 ,0){$\scriptstyle -5$}}

\end{picture}$$      
\vspace{0.5cm}
$$
\begin{picture}(0,0)(70,5) 
\multiput(-14,5.5)(35,0){5}{$\scriptstyle\bullet$}
\multiput(22.5,8)(35,0){4}{\line(-1,0){35} }
\put(58.5,14){\makebox(	0	,0){$\bigvee$}}
\put(128.5,14){\makebox(	0	,0){$\bigvee$}}
\put(23.5,1){\makebox(	0	,0){$\bigwedge$}}

\put(-11.5,25){\makebox(	0	,0){$\scriptstyle  4$}}
\put(	23.5	,25){\makebox( 0 ,0){$\scriptstyle 3$}}
\put(	58.5	,25){\makebox( 0 ,0){$\scriptstyle 2$}}
\put(	93.5	,25){\makebox( 0 ,0){$\scriptstyle 1 $}}
\put(	128.5	,25){\makebox( 0 ,0){$\scriptstyle 0$}}

\put(-11.5,-8){\makebox(	0	,0){$\scriptstyle 4 $}}
\put(	23.5	,-8){\makebox( 0 ,0){$\scriptstyle 0 $}}
\put(	58.5	,-8){\makebox( 0 ,0){$\scriptstyle 1 $}}
\put(	93.5	,-8){\makebox( 0 ,0){$\scriptstyle 2 $}}
\put(	128.5	,-8){\makebox( 0 ,0){$\scriptstyle 3 $}}

\end{picture}$$   
\vspace{0.5cm}
$$
\begin{picture}(0,0)(70,5) 
\multiput(-14,5.5)(35,0){5}{$\scriptstyle\bullet$}
\multiput(22.5,8)(35,0){4}{\line(-1,0){35} }

\put(-11.5,25){\makebox(	0	,0){$\scriptstyle  9$}}
\put(	23.5	,25){\makebox( 0 ,0){$\scriptstyle 8$}}
\put(	58.5	,25){\makebox( 0 ,0){$\scriptstyle 7$}}
\put(	93.5	,25){\makebox( 0 ,0){$\scriptstyle 6 $}}
\put(	128.5	,25){\makebox( 0 ,0){$\scriptstyle 5$}}

\end{picture}$$  
\end{ex}

\begin{rk}
\label{rk choice of abacus labels non rigid}
One could also view the characteristic $0$ diagram as the limiting case for $p$ going to $\infty$.
\end{rk}

In some sense, the choice of labels is arbitrary. We could for example perform a cyclic shift of the columns of the abacus and this definition would be just as good. In the case when the characteristic $p$ of the ground field is larger than the degree $d$ of the partition algebra, we will exploit this and the fact that the contents of boxes appearing in tableaux of a given partition are contained in a range which is smaller than $r+1$ and, therefore, less than or equal to $p$:

\begin{defn} For $p>r$, a partition $\mu=(\mu_1,\ldots,\mu_k)$ of $r$ and $\lambda \in \Lambda^+(r)$ with $\lambda \subseteq \mu$, we define the $\mu$-projection of the abacus $d(\lambda)$ as follows: It consists of a single runner with $p$ positions with labels above and below as follows: 
\vspace{0.5cm}
$$
\begin{picture}(0,0)(100,5) 
\multiput(-14,5.5)(35,0){6}{$\scriptstyle\bullet$}
\multiput(22.5,8)(35,0){5}{\line(-1,0){35} }

\put(-11.5,20){\makebox(	0	,0){$\scriptstyle  p-k-2$}}
\put(	23.5	,20){\makebox( 0 ,0){$\scriptstyle p-k-3$}}
\put(	58.5	,20){\makebox( 0 ,0){$\scriptstyle \cdots$}}
\put(	93.5	,20){\makebox( 0 ,0){$\scriptstyle -k+1 $}}
\put(	128.5	,20){\makebox( 0 ,0){$\scriptstyle -k$}}
\put(	163.5	,20){\makebox( 0 ,0){$\scriptstyle -k-1$}}

\put(-11.5,-5){\makebox(	0	,0){$\scriptstyle \delta+k+2 $}}
\put(	23.5	,-5){\makebox( 0 ,0){$\scriptstyle \delta+k+3 $}}
\put(	58.5	,-5){\makebox( 0 ,0){$\scriptstyle \cdots $}}
\put(	93.5	,-5){\makebox( 0 ,0){$\scriptstyle \delta+k-1 $}}
\put(	128.5	,-5){\makebox( 0 ,0){$\scriptstyle \delta+k $}}
\put(	163.5	,-5){\makebox( 0 ,0){$\scriptstyle  \delta+k+1$}}

\end{picture}$$   
\vspace{0.1cm}\\
where the labels above are integers and the labels below are in $\Z/p\Z$. For $i=1,\ldots,k$, we draw a $\bigvee$ above the line at position $\lambda_i-i$ and a $\bigwedge$ below the line at position $|\lambda| \pmod{p}$.

When $\mu$ is clear from the context and $p>r$, then we will usually identify the $\mu$-projection of $d(\lambda)$ with $d(\lambda)$.
For example, if $p=11$, $\delta=1$, $\lambda=(2,2,1,1)$ and $\mu=(3,3,2,1,1)$ (so that $k=5$), then  the $\mu$-projection of $d(\lambda)$ is
\vspace{0.6cm}
$$
\begin{picture}(0,0)(160,5) 
\multiput(-14,5.5)(35,0){11}{$\scriptstyle\bullet$}
\multiput(22.5,8)(35,0){10}{\line(-1,0){35} }

\put(	93.5	,13){\makebox( 0 ,0){$\scriptstyle \bigvee $}}
\put(	128.5	,13){\makebox( 0 ,0){$\scriptstyle \bigvee$}}
\put(	198.5	,13){\makebox( 0 ,0){$\scriptstyle \bigvee$}}
\put(	233.5	,13){\makebox( 0 ,0){$\scriptstyle \bigvee$}}
\put(	303.5	,13){\makebox( 0 ,0){$\scriptstyle \bigvee$}}

\put(	303.5	,3){\makebox( 0 ,0){$\scriptstyle \bigwedge$}}

\put(-11.5,25){\makebox(	0	,0){$\scriptstyle  4$}}
\put(	23.5	,25){\makebox( 0 ,0){$\scriptstyle 3$}}
\put(	58.5	,25){\makebox( 0 ,0){$\scriptstyle 2$}}
\put(	93.5	,25){\makebox( 0 ,0){$\scriptstyle 1 $}}
\put(	128.5	,25){\makebox( 0 ,0){$\scriptstyle 0$}}
\put(163.5,25){\makebox(	0	,0){$\scriptstyle  -1$}}
\put(	198.5	,25){\makebox( 0 ,0){$\scriptstyle -2$}}
\put(	233.5	,25){\makebox( 0 ,0){$\scriptstyle -3$}}
\put(	268.5	,25){\makebox( 0 ,0){$\scriptstyle -4 $}}
\put(	303.5	,25){\makebox( 0 ,0){$\scriptstyle -5$}}
\put(	338.5	,25){\makebox( 0 ,0){$\scriptstyle -6$}}

\put(-11.5,-8){\makebox(	0	,0){$\scriptstyle  8$}}
\put(	23.5	,-8){\makebox( 0 ,0){$\scriptstyle 9$}}
\put(	58.5	,-8){\makebox( 0 ,0){$\scriptstyle 10$}}
\put(	93.5	,-8){\makebox( 0 ,0){$\scriptstyle 0 $}}
\put(	128.5	,-8){\makebox( 0 ,0){$\scriptstyle 1$}}
\put(163.5,-8){\makebox(	0	,0){$\scriptstyle  2$}}
\put(	198.5	,-8){\makebox( 0 ,0){$\scriptstyle 3$}}
\put(	233.5	,-8){\makebox( 0 ,0){$\scriptstyle 4$}}
\put(	268.5	,-8){\makebox( 0 ,0){$\scriptstyle 5 $}}
\put(	303.5	,-8){\makebox( 0 ,0){$\scriptstyle 6$}}
\put(	338.5	,-8){\makebox( 0 ,0){$\scriptstyle 7$}}

\end{picture}$$   
\vspace{0.1cm}
\end{defn}
\begin{rk}
Notice that  since $\lambda \subseteq \mu$, the box with smallest content in $\lambda$ has content at least $-k$ (this corresponds to $\lambda_k=0$) and the largest possible content is $r-k-1 \leq p-k-2$ (which comes from the hook $\mu=(r-k+1,1^{k-1}))$. Thus, the $\mu$-projection of $d(\lambda)$ is well defined and there will never be a $\bigvee$ at label $-k-1$.
\end{rk}

We will frequently use shorthand notation for arrow diagrams. For example, the $\mu$-projection of $\lambda$ above will be either be represented as $\cdots \overset{1}{\bigvee} \bigvee \bigwedge \bigvee \bigvee \bigcirc \bigvee \cdots$ where $\bigvee$ stands for a column which only contains a $\bigvee$, and $\bigwedge$/$\bigtimes$/$\bigcirc$ represent columns containing only a $\bigwedge$, both a $\bigvee$ and a $\bigwedge$, or no arrows at all, respectively. The label at the top/bottom of a symbol is precisely the label above/below the line in that column and specifying one will determine all other labels. Sometimes we represent $\bigvee/\bigwedge/\bigtimes/\bigcirc$ in extended notation by $\raisebox{5pt}{$\vc$}/\raisebox{5pt}{$\wc$}/\raisebox{5pt}{$\xc$}/\raisebox{5pt}{$\oc$}$.

We will now connect the study of arrow diagrams with tableaux and the action of JM elements.
Our aim is to relate the property of having a common JM weight to the study of composition factors of the cell modules $\Delta(\lambda)$. One quite easily obtains the following relationship which is proved in Corollary 2.2 in \cite{JMpaper} in a more general context.

\begin{prop}
\label{prop same cf means same wt} Let $d\in \{r,r+1/2 \}$,
 $\lambda,\mu \in \Lambda^+(d)$ and suppose $L(\mu)$ is a composition factor of $\Delta(\lambda)$. Then $\lambda$ and $\mu$ have a common JM weight. If $\mu$ is a partition of $r$, then for every $\mu$-tableau $u$ there is a $\lambda$-tableau $t$ with $\wt(u)=\wt(t)$.
\end{prop}

The aim of this paper is to show that the converse is also true if $\mu$ is a partition of $r$ and the characteristic of the field does not divide $r$. By Theorem \ref{thm form of dec matrix}, this will also determine the decomposition number $[\Delta(\lambda):L(\mu)]$ for arbitrary $\mu$.

It turns out that having a common JM weight is quite restrictive. If we view a tableau as a construction plan for the corresponding partition, then the JM weight only leaves few possibilities for this construction. If the $k$th entry of the JM weight is $i$, that is, the JM element $L_k$ acts by $i$ on the canonical basis vector corresponding to the given tableau, then there are only at most two possibilities: If $k$ is not a natural number, then we may either add a box of content $i$ at the $k$th step or leave the partition unchanged if the size of the partition in question is $\delta-i$. If $k$ is a natural number, then we may either remove a box of content $\delta-i$ or leave the partition unchanged if its size is $i$.  The characteristic $0$ case is visualized in Figure \ref{figure directions to add arrows}.
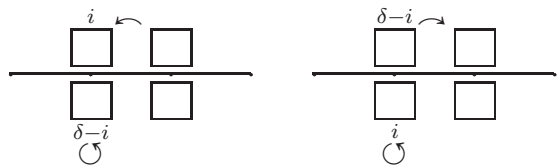
\begin{figure}
\caption{Possible ways to move arrows with JM eigenvalue $i$.}
\label{figure directions to add arrows}
\subfigure[From half to full degree]{$
\begin{picture}(110,60)
\put(30,20){\makebox{$\underset{\delta-i}{\boxed{\phantom{\times}}}$}}
\put(30,20){\makebox{$\boxed{\phantom{\times}}$}}
\put(30,40){\makebox{$\overset{i}{\boxed{\phantom{\times}}}$}}
\put(60,20){\makebox{$\boxed{\phantom{\times}}$}}
\put(60,40){\makebox{$\boxed{\phantom{\times}}$}}
\put(6,30){\makebox{$\cdot$}}
\put(6,30){\makebox{$\cdot$}}
\put(36,30){\makebox{$\cdot$}}
\put(36,30){\makebox{$\cdot$}}
\put(7,33){\line(1,0){30}}
\put(66,30){\makebox{$\cdot$}}
\put(66,30){\makebox{$\cdot$}}
\put(37,33){\line(1,0){30}}
\put(96,30){\makebox{$\cdot$}}
\put(96,30){\makebox{$\cdot$}}
\put(67,33){\line(1,0){30}}
\put(46,50){\makebox{$\curvearrowleft$}}
\put(33,0){\makebox{\rotatebox[origin=c]{0}{$\circlearrowleft$}}}
\end{picture} $}
\subfigure[From full to half degree]{$
\begin{picture}(110,60)
\put(30,40){\makebox{$\overset{\delta-i}{\boxed{\phantom{\times}}}$}}
\put(30,20){\makebox{$\boxed{\phantom{\times}}$}}
\put(30,40){\makebox{$\boxed{\phantom{\times}}$}}
\put(30,20){\makebox{$\underset{i}{\boxed{\phantom{\times}}}$}}
\put(60,20){\makebox{$\boxed{\phantom{\times}}$}}
\put(60,40){\makebox{$\boxed{\phantom{\times}}$}}
\put(6,30){\makebox{$\cdot$}}
\put(6,30){\makebox{$\cdot$}}
\put(36,30){\makebox{$\cdot$}}
\put(36,30){\makebox{$\cdot$}}
\put(7,33){\line(1,0){30}}
\put(66,30){\makebox{$\cdot$}}
\put(66,30){\makebox{$\cdot$}}
\put(37,33){\line(1,0){30}}
\put(96,30){\makebox{$\cdot$}}
\put(96,30){\makebox{$\cdot$}}
\put(67,33){\line(1,0){30}}
\put(46,50){\makebox{$\curvearrowright$}}
\put(33,0){\makebox{\rotatebox[origin=c]{0}{$\circlearrowleft$}}}
\end{picture}$}
\end{figure}
Here the boxes stand for potential positions of $\bigwedge$ and $\bigvee$ and the arrows indicate which way the arrows may be moved. The picture on the left is for the transition from half integer degree to full integer degree and the right one for the opposite direction.

If the characteristic is non-zero, but larger than the degree $d$ of the partition algebra, then the moves on the projected abacus are just as pictured above. For arbitrary primes $p$, we may not only add or remove boxes of content $i$ and $\delta-i$, respectively, but also of content $i+kp$ and $\delta-i+kp$ for $k \in \Z$. On the abacus, this translates into moving an arrow into a column containing the label $i$ and $\delta-i$, respectively. Thus the arrows in the illustration above are to be read as moving from one column to another, but the row is arbitrary.

Using these rules we may characterize partitions with a common JM weight. To state the next proposition, we first need another definition:

\begin{defn}
\label{defn of tau}
Given an arrow diagram $a$, we define $\tau(a)$ to be the same as $a$ except that $\bigwedge$ is moved up by $1$, that is, if the initial position is $i$, then we move it to position $i+1$. Similarly, $\tau \inv $ is defined to move $\bigwedge$ down by $1$.  
\end{defn}
It is easy to see that $\tau$ maps the set of arrow diagrams for $P_r(\delta)$ to the set of arrow diagrams for $P_{r-1/2}(\delta-1)$. Thus, $\tau$ is a graphical realization of the Morita equivalence between $P_{r}(\delta)$ and $P_{r-1/2}(\delta-1)$, see Theorem \ref{thm morita equiv of half with full}.

\begin{prop}
\label{prop common weight over arbitr p means same no of arrows}
Suppose $\lambda,\theta \in \Lambda^+(d)$ have a common JM weight. Then for all $k$ the $k$th columns of
\begin{enumerate}[(a)]
\item $d(\lambda)$ and $d(\theta)$ (if $d \in \N$)
\item  $\tau d(\lambda)$ and $\tau d(\theta)$ (if $d \notin \N$)
\end{enumerate}
contain the same number of arrows (that is, counting $\bigwedge$ and $\bigvee$).
\end{prop}

\begin{proof}
We will prove the result by induction on the degree of the partition algebra. Suppose the result is true for $d \in 1/2\N$. Let $u \in \UD(\lambda)$ and $t \in \UD(\theta)$ be the given tableau of the same weight. Let $\lambda'$ and $\theta'$ be the second last partitions in $u$ and $t$, respectively, and define $u'$ and $t'$ to be the tableaux obtained from $u$ and $t$ by omitting the last partition. Also, let $i$ be the last entry of $\wt(u)=\wt(t)$. Distinguish two cases:
\newline \textbf{Case 1:} $d \in \N$. In this case, we may either remove a box of content $\delta-i$ from the partitions in question or leave the partition unchanged if its size is $i$. Thus, to get from $d(\lambda')$ to $\tau d(\lambda)$, we first have to apply the operation in Figure \ref{figure directions to add arrows}(b) and then move the $\bigwedge$ to the right (that is, up) by one. We claim that in each case this amounts to moving an arrow from column $\delta-i$ to column $\delta-i-1$. For, if $\lambda'=\lambda$, then the only movement of arrows comes from applying $\tau$ which exactly moves the $\bigwedge$ from column $\delta-i$ to column $\delta-i-1$. If a box is removed from $\lambda'$ to obtain $\lambda$, then a $\bigvee$ moves from column $\delta-i$ to column $\delta-i-1$ but also a $\bigwedge$ moves down by $1$ (since the size of the partition changes). This, however, is precisely offset  by the application of $\tau$ so that overall only the $\bigvee$ moves as claimed. Notice that exactly the same argument applies to $d(\theta')$ and $d(\theta)$ so that $\tau d(\lambda)$ and $\tau d(\theta)$ have the same number of arrows in each column.
\newline \textbf{Case 2:} $d \notin \N$. In this case, we may either have $\lambda'=\lambda$ if the size of $\lambda'$ is $\delta-i$ or add a box of content $i$ to $\lambda'$ to get $\lambda$. Thus,  to get from $\tau d(\lambda')$ to $d(\lambda)$, we have to first apply $\tau\inv$  and then  the operations from Figure \ref{figure directions to add arrows}(a). We claim that in either case an arrow will be moved from column $i-1$ to $i$. If a box is added to $\lambda'$, this follows by exactly the same argument as in Case 1. If $\lambda'=\lambda$, then there must be a $\bigwedge$ in column $i$ of $d(\lambda')=d(\lambda)$ which means that the $\bigwedge$ was in column $i-1$ of $\tau d(\lambda')$, as claimed. Again, the same argument works for $\theta$ and $\theta'$ and we are done.
\end{proof}

In the special case when $\mu$ is a partition of $r$ and the characteristic of the field is $0$ or larger than $r$, we can even say more:

\begin{prop}
\label{prop common jm weight implies certain form}
Suppose $\lambda,\mu \in \Lambda^+$ with $\lambda \neq \mu$, $p=0$ or $p>r$ and $\mu$ is a partition of $r$. Then $\lambda$ and $\mu$ have a common JM weight if and only if the arrow diagrams $d(\lambda)$ and $d(\mu)$ are the same except for one configuration of $l$ neighbouring columns of the following form:\\
  $\begin{array}{llllll}(a) \ \mu \in \Lambda^+(r+1/2) \hspace*{0.1cm}&
\mu:  & \qc \vc \oc  \cdots \oc  \wc \oc 
(b) \ \mu \in \Lambda^+(r) \hspace*{0.1cm}&
\mu:  &   \vc \oc \cdots \oc \wc     \\ 
\ & \lambda:   & \underset{\textstyle \raisebox{-1.9mm}{$\bigwedge$}}{?} \oc  \oc \cdots \oc \oc \vc  & 
 \lambda:   & \wc \oc \cdots \oc \vc
\end{array}$ \newline where $(a)$ also encompasses the special case  $\begin{array}{ccccc}
\mu: & \qc \xc \oc  & \hspace*{-0.2cm} &
\lambda: & \underset{\textstyle \raisebox{-1.9mm}{$\bigwedge$}}{?}  \oc \vc 
\end{array}$ and $?$ is either equal to $\bigvee$ or $\bigcirc$.
\end{prop}

\begin{proof} 
We will start with the case when $\lambda$ and $\mu$ have a common JM weight so that we have to show that they are of the required form. For $r=1$ this is clearly true. So suppose it is true for degree $d \in 1/2 \N$. We will show it is true for $d+1/2$ as well. Let $u \in \UD(\mu)$ and $t \in \UD(\lambda)$ be the tableaux of the same weight. As before, denote by $\mu'$ and $\lambda'$ the second last partition in $u$ and $t$, respectively, and suppose the last entry of the weight of $u$ and $t$ is $i$. By induction hypothesis, $\lambda'$ and $\mu'$ have the claimed form.  Distinguish the following four cases:
\newline \textbf{Case 1: }\textit{$\mu'=\lambda'$ and $d \in \N$.} In this case, we move from the full integer degree to the half integer degree case which corresponds to Figure \ref{figure directions to add arrows}(b). Now in the $\mu$-tableau $u$ we may only add boxes or leave the partition unchanged when moving from one step of $u$ to the next. In this case, we only have the choice $\mu'=\mu$. In particular, there is a $\bigwedge$ at position $i$ of $\mu'=\lambda'$. If $\lambda=\lambda'$ as well, then we are done so assume this is not the case. Thus, we must remove a box of content $\delta-i$ from $\lambda'$ and since $\lambda'=\mu'$ it follows that $\mu'$ also has a removable box of content $\delta-i$. Thus we can depict the situation as follows:$$\begin{array}{cccc}
u: & d(\mu')=\cdots  \qc \raisebox{3pt}{$\overset{\delta-i}{\xc} \oc$} \cdots & - & d(\mu)=\cdots \qc \raisebox{3pt}{$\overset{\delta-i}{\xc} \oc$}\cdots \\
t: & d(\lambda')= \cdots  \qc \raisebox{3pt}{$\overset{\delta-i}{\xc} \oc$}  \cdots & - & d(\lambda)=\cdots \underset{\textstyle \raisebox{-1.9mm}{$\bigwedge$}}{?}  \raisebox{3pt}{$\overset{\delta-i}{\oc} \vc$}\cdots
\end{array}.$$ All other arrows are unchanged and so $\mu$ and $\lambda$ indeed have the claimed form. 
\newline \textbf{Case 2: }\textit{$\mu'=\lambda'$ and $d \notin \N$.} Using a similar reasoning as in Case 1, we must have that a box of content $i$ is added to $\mu'$ and we only have to consider the case $\lambda=\lambda'$ which implies that there is a $\bigwedge$ at position $\delta-i$ in both $\lambda'$ and $\mu'$. Therefore, the situation may be depicted as follows:
$$\begin{array}{ccccccccc}
u: & \cdots \raisebox{3pt}{$\overset{i}{\wc} \vc $} \cdots& - & \cdots\raisebox{3pt}{$\overset{i}{\vc} \wc $} \cdots & \ &
t: & \cdots \raisebox{3pt}{$\overset{i}{\wc} \vc $} \cdots & - & \cdots \raisebox{3pt}{$\overset{i}{\wc} \vc $}\cdots
\end{array}$$ which is of the required form.
\newline \textbf{Case 3: }\textit{$\mu'\neq \lambda'$ and $d \in \N$.} Similar as in Case 1, we must have $\mu'=\mu$. However, we have to consider both the case when $\lambda=\lambda'$ and the case when a box is removed at the last step of $t$. But if $\lambda=\lambda'$, then $|\lambda|=\delta-i=|\mu|$ which forces $\lambda=\mu$ since $p>r$ and $\mu$ is a partition of $r$. So we can assume that a box is removed at the last step of $t$. By induction hypothesis $\mu'$ and $\lambda'$ must be the same except a column configuration as in $(b)$. Since $\mu'=\mu$, the only box which we may remove from $\lambda'$ is the $\bigvee$ in the right-most drawn column which is in the same column as the $\bigwedge$ in $\mu'$. In particular,  $\lambda$ and $\mu$ will be precisely configured as in $(a)$ by Proposition \ref{prop common weight over arbitr p means same no of arrows} which forces the right-most drawn column of $\mu$ to contain no arrows.  
\newline \textbf{Case 4: }\textit{$\mu'\neq \lambda'$ and $d \notin \N$.} In this case, we must add a box to $\mu'$ to obtain $\mu$. Start with the case, when $\lambda=\lambda'$. Because $\lambda'$ and $\mu'$ are configured as in $(a)$, and we have to move arrows as in Figure \ref{figure directions to add arrows}(a), there is only one choice to move an arrow in $\mu'$ to obtain $\mu$: We have to move the $\bigvee$  of $\mu'$, which is in the column to the right of the column containing the $\bigwedge$ in $\lambda'$,  up by $1$, that is to the left. This does, in particular, force $?$ in $\mu'$ (and hence in $\lambda'=\lambda$ by Proposition \ref{prop common weight over arbitr p means same no of arrows}) to be $\bigcirc$ so that we precisely obtain a configuration as in $(b)$. If  a box is added to $\lambda'$ to obtain $\lambda$, then the corresponding arrows of $\lambda'$ and $\mu'$ have the same label (i.e.~content) since $p>r$ and must hence lie outside the configuration drawn in $(a)$. This will leave the configuration of arrows outside the columns drawn in $(a)$ the same for both $\lambda$ and $\mu$ and within the column configuration both $\bigwedge$ will move by one to the right. This yields precisely a configuration as in $(b)$.

Notice that all four cases work exactly the same for the $\mu$-projected abacus in characteristic $p>r$. This follows because the projection was chosen precisely in such a way that the active area, that is the set of contents of the boxes of $\mu$, is drawn and there is one additional "space" on the right of the diagram. Thus, the operations depicted in Cases 1-4 never leave the abacus projection. 

It remains to show the converse of the statement. However, in this case we can simply work backwards in each of the cases above which defines the required tableaux of the same weight.
\end{proof}

The proposition motivates the following definition:
\begin{defn}
\label{defn sigma action of arrow diag}
Define an action $\sigma$ on a characteristic $0$ diagram as follows: If $\bigwedge$ is in column $j$ and $i \leq j$ is maximal such that it contains a $\bigvee$, then move $\bigwedge$ to column $i$ and the $\bigvee$ in column $i$ to column $j$. Thus, $\sigma$ simply changes a neighbouring $\bigvee \bigwedge$ pair to a $\bigwedge \bigvee$ pair.
\end{defn}

Thus condition $(b)$ in the proposition could just be reformulated as $d(\lambda)=\sigma d(\mu)$ and condition $(a)$ as $d(\lambda)=\tau \inv \sigma \tau d(\mu)$.

An important corollary of the proof of Proposition \ref{prop common jm weight implies certain form} is that decomposition multiplicities must always be either $0$ or $1$ if $p>r$:
\begin{cor} \label{cor dec multiplicities are 0 or 1}
 Suppose $p>r$ and let $\lambda,\mu \in \Lambda^+(d)$ with $\mu$ a partition of $r$. Then:
\begin{enumerate}[(a)]
\item If $\lambda$ and $\mu$ have a common JM weight, then for every $\mu$-tableau $u$ there is at most one $\lambda$-tableau $t$ with $\wt(u)=\wt(t)$.
\item $[\Delta(\lambda):L(\mu)] \in \{ 0,1\}$. 
\end{enumerate}
\end{cor}
\begin{proof}
The first part follows since in the proof of Proposition \ref{prop common jm weight implies certain form}, there were never two choices to move arrows in $t$. This immediately implies the second part since if  $[\Delta(\lambda):L(\mu)] >1$, then $\Delta(\lambda)$ contains at least two copies of $L(\mu)=\Delta(\mu)$. Thus, the generalized eigenspace corresponding to $\wt(t)$ contains at least two elements which means that for every $\mu$-tableau $u$, there would be at least two different $\lambda$-tableau of the same weight, see the proof of Corollary 3.5 in \cite{JMpaper}. 
\end{proof}

\section{Refined induction and restriction}
We will introduce and study refined induction and restriction functors which will be needed for an inductive proof of Theorem \ref{thm determination of dec nos}.

\begin{defn}
 For a $P_d(\delta)$-module $M$ and $i \in \Z$, define $\iind_d M$ to be the projection onto the generalized eigenspace of the action of the JM-element $L_{d+1/2}$ on $\ind_d M$ with eigenvalue $i$. Similarly, $\ires_{d+1/2} M$ is the generalized eigenspace of the action of the JM-element $L_{d}$ on $\res_{d+1/2} M$ with eigenvalue $i$. We usually omit $d$ if it is clear from the context.
\end{defn}

These maps have  similar properties as the corresponding functors for symmetric groups:

\begin{prop}
The maps \begin{align*}
\iind: P_d(\delta)\text{--mod} \to P_{d+1/2}(\delta)\text{--mod}  \hspace*{0.8cm} \ires: P_{d+1/2}(\delta)\textrm{--mod} \to P_{d}(\delta)\text{--mod}
\end{align*} are functorial and $\iind$ is a left-adjoint to $\ires$.

\end{prop}
The proof is exactly as in the symmetric group case, see e.g.~6.4 in \cite{Mathasbook}. All that is required, is that the element $z=L_{1/2} + \ldots L_d$, that is the sum of all JM-elements, is central, see Theorem \ref{thm sum of JMs central}.

For cell modules, we can deduce the following induction/restriction rules from Theorem \ref{thm restriction of cells} by projecting onto the corresponding JM weight space:

\begin{theorem} \label{thm ires of cells} Suppose $\lambda \in \Lambda^+(d)$ with $d \in \{r,r+1/2 \}$ and that the characteristic $p$ of the ground field is zero or larger than $d$. Then there are exact sequences as follows
\begin{enumerate}[(a)]
\item $\displaystyle 0 \longrightarrow   \Delta(\lambda-\epsilon)  \longrightarrow \ires_r  \Delta(\lambda) \longrightarrow \Delta(\lambda) \longrightarrow 0,$ where $\epsilon$ is a removable box of content $i$ and the term on the right is zero unless $|\lambda|=\delta-i$.
\item  $\displaystyle 0 \longrightarrow\Delta(\lambda)\longrightarrow \ires_{r+1/2} \Delta(\lambda) \longrightarrow \Delta(\lambda+\epsilon) \longrightarrow 0, $ where $\epsilon$ is an addable box of content $\delta - i$ and  the term on the left is zero unless  $|\lambda|=i$. 
\item $\displaystyle 0 \longrightarrow   \Delta(\lambda-\epsilon)\longrightarrow \iind_r \Delta(\lambda) \longrightarrow \Delta(\lambda) \longrightarrow 0,$ where $\epsilon$ is a removable box of content $\delta-i$ and the term on the right is zero unless  $|\lambda|=i$.
\item $ \displaystyle 0 \longrightarrow\Delta(\lambda)\longrightarrow \iind_{r-1/2} \Delta(\lambda) \longrightarrow  \Delta(\lambda+\epsilon) \longrightarrow 0, $  where $\epsilon$ is an addable box of content $ i$ and the term on the left is zero unless  $|\lambda|=\delta-i$. assuming
\end{enumerate}
In each case, we set $\Delta(\lambda\pm \epsilon)=0$ if no addable/removable box $\epsilon$ with the required content exists.
\end{theorem}
\begin{rk}
Notice that since $p>r$ or $p=0$, there is always at most one addable and at most one removable box of a given content.
\end{rk}
Figure \ref{figure directions to move for ires and iind} summarizes the theorem in a similar way as Figure \ref{figure directions to add arrows}.
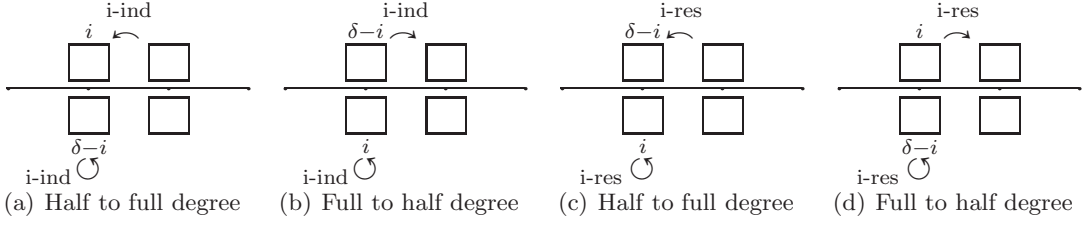
\begin{figure}
\caption{Possible ways to move arrows when $i$-restricting/$i$-removing.}
\label{figure directions to move for ires and iind}
\subfigure[Half to full degree]{$
\begin{picture}(100,70)
\put(44,60){\makebox{$\scriptstyle{\iind}$}}
\put(14,-3){\makebox{$\scriptstyle{\iind}$}}
\put(30,20){\makebox{$\underset{\delta-i}{\boxed{\phantom{\times}}}$}}
\put(30,20){\makebox{$\boxed{\phantom{\times}}$}}
\put(30,40){\makebox{$\overset{i}{\boxed{\phantom{\times}}}$}}
\put(60,20){\makebox{$\boxed{\phantom{\times}}$}}
\put(60,40){\makebox{$\boxed{\phantom{\times}}$}}
\put(6,30){\makebox{$\cdot$}}
\put(6,30){\makebox{$\cdot$}}
\put(36,30){\makebox{$\cdot$}}
\put(36,30){\makebox{$\cdot$}}
\put(7,33){\line(1,0){30}}
\put(66,30){\makebox{$\cdot$}}
\put(66,30){\makebox{$\cdot$}}
\put(37,33){\line(1,0){30}}
\put(96,30){\makebox{$\cdot$}}
\put(96,30){\makebox{$\cdot$}}
\put(67,33){\line(1,0){30}}
\put(46,50){\makebox{$\curvearrowleft$}}
\put(33,0){\makebox{\rotatebox[origin=c]{0}{$\circlearrowleft$}}}
\end{picture} $}
\subfigure[Full to half degree]{$
\begin{picture}(100,70)
\put(44,60){\makebox{$\scriptstyle{\iind}$}}
\put(14,-3){\makebox{$\scriptstyle{\iind}$}}
\put(30,40){\makebox{$\overset{\delta-i}{\boxed{\phantom{\times}}}$}}
\put(30,20){\makebox{$\boxed{\phantom{\times}}$}}
\put(30,40){\makebox{$\boxed{\phantom{\times}}$}}
\put(60,20){\makebox{$\boxed{\phantom{\times}}$}}
\put(30,20){\makebox{$\underset{i}{\boxed{\phantom{\times}}}$}}
\put(60,40){\makebox{$\boxed{\phantom{\times}}$}}
\put(6,30){\makebox{$\cdot$}}
\put(6,30){\makebox{$\cdot$}}
\put(36,30){\makebox{$\cdot$}}
\put(36,30){\makebox{$\cdot$}}
\put(7,33){\line(1,0){30}}
\put(66,30){\makebox{$\cdot$}}
\put(66,30){\makebox{$\cdot$}}
\put(37,33){\line(1,0){30}}
\put(96,30){\makebox{$\cdot$}}
\put(96,30){\makebox{$\cdot$}}
\put(67,33){\line(1,0){30}}
\put(46,50){\makebox{$\curvearrowright$}}
\put(33,0){\makebox{\rotatebox[origin=c]{0}{$\circlearrowleft$}}}
\end{picture}$}
\subfigure[Half to full degree]{$
\begin{picture}(100,70)
\put(44,60){\makebox{$\scriptstyle{\ires}$}}
\put(14,-3){\makebox{$\scriptstyle{\ires}$}}
\put(30,20){\makebox{$\underset{i}{\boxed{\phantom{\times}}}$}}
\put(30,20){\makebox{$\boxed{\phantom{\times}}$}}
\put(30,40){\makebox{$\overset{\delta-i}{\boxed{\phantom{\times}}}$}}
\put(60,20){\makebox{$\boxed{\phantom{\times}}$}}
\put(60,40){\makebox{$\boxed{\phantom{\times}}$}}
\put(6,30){\makebox{$\cdot$}}
\put(6,30){\makebox{$\cdot$}}
\put(36,30){\makebox{$\cdot$}}
\put(36,30){\makebox{$\cdot$}}
\put(7,33){\line(1,0){30}}
\put(66,30){\makebox{$\cdot$}}
\put(66,30){\makebox{$\cdot$}}
\put(37,33){\line(1,0){30}}
\put(96,30){\makebox{$\cdot$}}
\put(96,30){\makebox{$\cdot$}}
\put(67,33){\line(1,0){30}}
\put(46,50){\makebox{$\curvearrowleft$}}
\put(33,0){\makebox{\rotatebox[origin=c]{0}{$\circlearrowleft$}}}
\end{picture} $}
\subfigure[Full to half degree]{$
\begin{picture}(100,70)
\put(44,60){\makebox{$\scriptstyle{\ires}$}}
\put(14,-3){\makebox{$\scriptstyle{\ires}$}}
\put(30,40){\makebox{$\overset{i}{\boxed{\phantom{\times}}}$}}
\put(30,20){\makebox{$\boxed{\phantom{\times}}$}}
\put(60,20){\makebox{$\boxed{\phantom{\times}}$}}
\put(30,40){\makebox{$\boxed{\phantom{\times}}$}}
\put(30,20){\makebox{$\underset{\delta-i}{\boxed{\phantom{\times}}}$}}
\put(60,40){\makebox{$\boxed{\phantom{\times}}$}}
\put(6,30){\makebox{$\cdot$}}
\put(6,30){\makebox{$\cdot$}}
\put(36,30){\makebox{$\cdot$}}
\put(36,30){\makebox{$\cdot$}}
\put(7,33){\line(1,0){30}}
\put(66,30){\makebox{$\cdot$}}
\put(66,30){\makebox{$\cdot$}}
\put(37,33){\line(1,0){30}}
\put(96,30){\makebox{$\cdot$}}
\put(96,30){\makebox{$\cdot$}}
\put(67,33){\line(1,0){30}}
\put(46,50){\makebox{$\curvearrowright$}}
\put(33,0){\makebox{\rotatebox[origin=c]{0}{$\circlearrowleft$}}}
\end{picture}$}
\end{figure}

For the simple modules, we can only describe the socle in full generality which will, however, be sufficient for our purposes. For the statement of the Theorem, recall the notation from Proposition \ref{prop common jm weight implies certain form}.

\begin{theorem}
\label{thm ires of simples}
\begin{enumerate}[(a)]
\item If $\ires \Delta(\lambda)=0$, then $\ires L(\lambda)=0$.
\item If $\ires \Delta(\lambda)= \Delta(\lambda')$ and $\iind \Delta(\lambda')=\Delta(\lambda)$, then $\ires L(\lambda)=L(\lambda')$.
\item Let $\lambda^+=\cdots \overset{i}{\bigvee}  \bigwedge \cdots \in \Lambda^+(r)$ and $\lambda^-=\cdots \overset{i}{\bigwedge} \bigvee \cdots \in \Lambda^+(r)$, then $\ires L(\lambda^+)=L(\lambda^-)$ and $\ires L(\lambda^-)=0$.
\item Let $\lambda^+=\cdots \raisebox{3pt}{$\qc \overset{\delta-i}{\xc}  \oc$} \cdots \in \Lambda^+(r+1/2)$ and $\lambda^-=\cdots \raisebox{3pt}{$\underset{\textstyle \raisebox{-1.9mm}{$\bigwedge$}}{?} \overset{\delta-i}{\oc} \vc$} \cdots \in \Lambda^+(r+1/2)$, then $\ires L(\lambda^+)=L(\lambda^-)$ and $\ires L(\lambda^-)=0$.
\item If $\lambda=\cdots \raisebox{4pt}{$\qc \overset{i}{\xc} \oc$} \cdots \in \Lambda^+(r)$, then the socle of $L(\lambda)$ is isomorphic to \begin{center}
$L(\cdots \raisebox{4pt}{$\qc \overset{i}{\xc} \oc$} \cdots) \oplus L(\cdots \raisebox{4pt}{$\underset{\textstyle \raisebox{-1.9mm}{$\bigwedge$}}{?} \overset{i}{\oc} \vc$} \cdots)$.
\end{center}
\item If $\lambda=\cdots \overset{\delta-i}{\bigwedge} \bigvee \cdots \in \Lambda^+(r+1/2)$, then the socle of $L(\lambda)$ is isomorphic to \begin{center}
$L(\cdots \overset{\delta-i}{\bigwedge}  \bigvee \cdots) \oplus L(\cdots \overset{\delta-i}{\bigvee}  \bigwedge \cdots)$.
\end{center}
\end{enumerate}
\end{theorem}

\begin{rk}
Notice that all possibilities for $\lambda$ are covered in the theorem. Part $(a)$ covers the case when there are no arrows which can be moved so that the induced JM eigenvalue is $i$. Parts $(b),(c)$ and $(d)$ cover the case if there is one possibility to move arrows in $\lambda$ with induced JM weight $i$ and Parts $(e)$ and $(f)$ if there are two possibilities. Part $(b)$ and Parts $(c)$ and $(d)$ then distinguish how many ways there are to move arrows back, namely one or two. 
\end{rk}

\begin{proof}
We will repeatedly use the following facts: Firstly,
by exactness of $\ires$, the surjection $\Delta(\lambda) \to L(\lambda)$ from Theorem \ref{thm classification of simples}  yields \begin{align} \label{eqn ires l quotient of ires delta}
\ires \Delta(\lambda) \longrightarrow \ires L(\lambda) \longrightarrow 0
\end{align}
Also, if $L(\tau)$ occurs in the socle of $\ires L(\lambda)$, then $\Hom(\Delta(\tau),\ires L(\lambda)) \neq 0$ so that adjointness implies \begin{align}
\Hom(\iind \Delta(\tau), L(\lambda)) \neq 0
\end{align} Lastly, the head of $\iind \Delta(\tau)$ may only contain the modules  $L(\tau)$, $L(\tau+\epsilon)$ and $L(\tau-\epsilon')$ by Theorem \ref{thm ires of cells}.  Since $L(\lambda)$ is simple we may conclude that one of these must be equal to $L(\lambda)$ if $L(\tau)$ occurs in the socle of $\ires L(\lambda)$. More precisely, we must either have:
\begin{itemize}
\item $\tau=\lambda+\epsilon$ provided $\lambda \in \Lambda^+(r+1/2)$ and $c(\epsilon)=\delta-i$ or
\item $\tau=\lambda-\epsilon$ provided $\lambda \in \Lambda^+(r)$ and $c(\epsilon)=i$ or
\item $\tau=\lambda$ provided $\lambda \in \Lambda^+(r+1/2)$ and $|\lambda|=i$ or
\item $\tau=\lambda$ provided $\lambda \in \Lambda^+(r)$ and $|\lambda|=\delta-i$.
\end{itemize}
Part $(a)$ follows immediately from (\ref{eqn ires l quotient of ires delta}), so consider $(b)$. Again by (\ref{eqn ires l quotient of ires delta}), $\ires L(\lambda)$ is a quotient of $\ires \Delta(\lambda)=\Delta(\lambda')$. In particular, its head must be $L(\lambda')$ or zero. We will examine the socle of $\ires L(\lambda)$ more closely. As shown above, if a module $L(\tau)$ occurs in the socle of $\ires L(\lambda)$, then $\tau$ must be equal to one of the possibilities above. If $\tau=\lambda+\epsilon$ with $c(\epsilon)=\delta-i$, then this means that $\lambda$ has an addable box of content $\delta-i$ so that it follows with the assumption that $\ires \Delta(\lambda)=\Delta(\lambda')$ that $\ires \Delta(\lambda)=\Delta(\lambda+\epsilon)$ and thus $\tau=\lambda'$. Similarly, if $\tau=\lambda-\epsilon$, then $\ires \Delta(\lambda)=\Delta(\lambda-\epsilon)$ and again $\tau=\lambda'$. With exactly the same argument, one easily shows that if $\tau=\lambda$ then $\ires \Delta(\lambda)=\Delta(\lambda)$ in both the case $\lambda \in \Lambda^+(r)$ and $\lambda \in \Lambda^+(r+1/2)$ so that here again $\tau=\lambda'$. Therefore, if a module occurs in the socle of $\ires L(\lambda)$, then it must be isomorphic to $L(\lambda')$. Since $L(\lambda')$ is also its simple head and $\ires L(\lambda)$ is a quotient of $\Delta(\lambda')$ which contains $L(\lambda')$ only once as a composition factor by Theorem \ref{thm form of dec matrix}, we can conclude that either $\ires L(\lambda)=L(\lambda')$ or zero. But since $\iind \Delta(\lambda')=\Delta(\lambda)$, composing with the surjection from $\Delta(\lambda)$ onto its head $L(\lambda)$ yields $0 \neq \Hom(\iind \Delta (\lambda'),L(\lambda))=\Hom(\Delta(\lambda'), \ires L(\lambda))$ and so $\ires L(\lambda)$ is non-zero, as required.

Now consider Parts $(c)$ and $(d)$ simultaneously. We will write $\lambda$ for either $\lambda^+$ and $\lambda^-$.  In each case, $\ires \dla = \dlap$ with $\lambda'= \lambda^-$ for $\lambda^- \in \Lambda^+(r)$ and $\lambda'= \lambda^+$ for $\lambda^+ \in \Lambda^+(r+1/2)$. Thus, $\ires L(\lambda)$ is a quotient of $\Delta(\lambda')$. We will again examine the socle of $\ires \lla$ more closely. If $\lambda \in \Lambda^+(r)$, then $\tau$ can either be $\lambda-\epsilon$ (with $c(\epsilon)=i$) or $\lambda$ (with $|\lambda|=\delta-i$). However, $\lambda^-$ has no removable box of content $i$ (so $\tau=\lambda^+-\epsilon=\lambda^-$) and $|\lambda^+|\neq \delta-i$ (so $\tau=\lambda^-$). In each case, we obtain that $\tau=\lambda^-=\lambda'$. Similarly, if $\lambda \in \Lambda^+(r+1/2)$, then either $\tau=\lambda+\epsilon$ (with $c(\epsilon)=\delta-i$) or $\tau=\lambda$ (with $|\lambda|=i$). Again, in each case $\tau=\lambda^+=\lambda'$. Therefore, just as in $(b)$,  $\ires \lla$ is either $L(\lambda')$ or zero.

We claim that in both cases $[\Delta(\lambda^-):L(\lambda^+)] \neq 0$. In the case, when $\lambda \in \Lambda^+(r)$, this follows from Theorem \ref{thm differ by one box then common CF} (and, of course, Theorem \ref{thm form of dec matrix}) since $c(\epsilon)=i=\delta-(\delta-i+1)+1=\delta-|\lambda^+|+1$ (considered modulo $p$ if $p>r$). In the case when $\lambda \in \Lambda^+(r+1/2)$ we can deduce it from the integer degree case: Firstly, we may without loss of generality restrict to the case of the partition algebra of degree $|\lambda^+|+1/2$ by Theorem \ref{thm form of dec matrix}. Thus, $\Delta(\lambda^+)$ is an inflated symmetric group module and therefore simple as $p$ is larger than $r$. Since every partition (except the empty partition) has at least two addable boxes and since $\lambda^+$ and $\lambda^-$ must be identical outside the two drawn columns, there is an undrawn $\bigvee$ outside the drawn columns which is both addable in $\lambda^+$ and $\lambda^-$ and has content $j$ say. This means that  there is an arrow in both $\lambda^+$ and $\lambda^-$ labelled $j-1$ with no $\bigvee$ to the left of it so that $j-1 \neq \delta-i,\delta-i-1,\delta-i-2$.  If $\epsilon$ and $\epsilon'$ are the corresponding addable boxes of content $j$, then  $\textrm{j-res} \ \Delta(\lambda^++\epsilon)=\Delta(\lambda^+) $ and $\textrm{j-res} \ \Delta(\lambda^-+\epsilon')=\Delta(\lambda^-) $  since $|\lambda^++\epsilon|=i+1 \neq \delta-j$ (otherwise $j-1=\delta-i-2$) and $|\lambda^-+\epsilon'|=i \neq \delta-j$. By Theorem \ref{thm differ by one box then common CF}, we conclude that $[\Delta(\lambda^-+\epsilon'):L(\lambda^++\epsilon)] \neq 0$ and so there is a non-zero map $\Delta(\lambda^++\epsilon) \to \Delta(\lambda^-+\epsilon')/M$ for some module $M$ and the map is injective since $\Delta(\lambda^++\epsilon)$ is simple.  Applying $\textrm{j-res}$ thus yields an injective, and hence non-zero, map $\Delta(\lambda^+) \to \Delta(\lambda^-)/\textrm{j-res} \ M$, as claimed.

By Theorem \ref{thm ires of cells}, there is a surjection $\iind \Delta(\lambda') \to \Delta(\lambda^+)$ in both the integer and half integer case. Composing with the surjection $\Delta(\lambda^+) \to L(\lambda^+)$, we obtain $$0 \neq \Hom(\iind \Delta(\lambda'), L(\lambda^+)) = \Hom (\Delta(\lambda'), \ires L(\lambda^+)). $$ In particular, $\ires L(\lambda^+)$ is non-zero and therefore isomorphic to $L(\lambda')$ as shown above. Now $\ires \Delta(\lambda^-)=\Delta(\lambda')$  as it is equal to $\Delta(\lambda^-)$ in the integer and $\Delta(\lambda^+)$ in the half integer degree case. But because $\Delta(\lambda^-)$ has both $L(\lambda^+)$ and $L(\lambda^-)$ as composition factors, it follows that $\Delta(\lambda')$ has $\ires L(\lambda^+)=L(\lambda')$ and $\ires L(\lambda^-)$ as quotients by exactness of $\ires$.  It follows that $\ires L(\lambda^-)=0$ since otherwise $[\Delta(\lambda'):L(\lambda')] >1$ - a contradiction to Theorem \ref{thm form of dec matrix}.

It remains to consider Parts $(e)$ and $(f)$. Again, if a module $L(\tau)$ occurs in the socle of $\ires L(\lambda)$ then $\tau$ must be one of the possibilities listed at the beginning of the proof and this yields precisely the partitions in the theorem. It remains to show that all of them do occur and with multiplicity $1$. One checks that for all possibilities for $\tau$, we always have $\iind \Delta(\tau)= \Delta(\lambda)$. Thus, we get \begin{align*}
1 = \dim \Hom(\Delta(\lambda),L(\lambda))= \dim \Hom(\iind \Delta(\tau),L(\lambda)) = \dim \Hom(\Delta(\tau),\ires L(\lambda))
\end{align*} so each possibility does occur and with multiplicity $1$.
\end{proof}

\section{Decomposition numbers in non-dividing characteristic}
We are finally in a position to prove a characterisation of decomposition numbers:

\begin{theorem}
\label{thm determination of dec nos} Let $d\in \{r,r+1/2\}$ and
suppose the characteristic of the ground field is $0$ or larger than $r$. If $\mu$ is a partition of $r$ and $\lambda \in \Lambda^+(d)$, then $L(\mu)$ is a composition factor of $\Delta(\lambda)$ if and only if there is a $\mu$-tableau $u$ and a $\lambda$-tableau $t$ with $\wt(u)=\wt(t)$ or, equivalently, if and only if for all  $\mu$-tableau $u$ there is a $\lambda$-tableau $t$ with $\wt(u)=\wt(t)$.
\end{theorem}

\begin{proof}
We have already shown two of the required three implications in Proposition \ref{prop same cf means same wt}. We will show the remaining direction by induction on the degree. For $r=1$ the result is trivially true. Suppose it is true in degree $d \in \{r-1/2,r\}$. Let $u$ be a $\mu$-tableau and $t$ be a $\lambda$-tableau with $\wt(u)=\wt(t)$. Let $\mu'$ and $\lambda'$ be the second-last partitions in $u$ and $t$, respectively, and denote by $u'$ and $t'$ the tableaux obtained by removing the last partition in $u$ and $t$.

Since $\lambda$ and $\mu$ have a common JM weight, Proposition \ref{prop common jm weight implies certain form} implies that they are identical except possibly at one configuration of columns as follows:\\
   $\begin{array}{llllll} \textrm{If} \ \mu \in \Lambda^+(r+1/2): \hspace*{0.1cm}&
\mu:  & \qc \vc \oc  \cdots \oc  \wc \oc  \ \
\textrm{If} \ \mu \in \Lambda^+(r): \hspace*{0.1cm}&
\mu:  &   \vc \oc \cdots \oc \wc     \\ 
\ & \lambda:   & \underset{\textstyle \raisebox{-1.9mm}{$\bigwedge$}}{?} \oc  \oc \cdots \oc \oc \vc  & 
 \lambda:   & \wc \oc \cdots \oc \vc
\end{array}$ \\ In particular, outside of these columns only arrows of the form  $\bigvee$ occur and the $\bigvee$ occur in the same columns for both. Notice that in the case when the deviating columns are of the form 
$\raisebox{4pt}{$\qc \xc \oc$}$ and $\raisebox{4pt}{$\underset{\textstyle \raisebox{-1.9mm}{$\bigwedge$}}{?}  \oc \vc $}$,  we are done by Theorem \ref{thm differ by one box then common CF} (and its generalization to the half integer case, see the proof of Theorem \ref{thm ires of simples}). 

We claim that $\ires \dmu= \dmup$ and $\ires \dla  = \dlap$ in each case. If $\mu \in \Lambda^+(r+1/2)$, then we must necessarily have $\mu=\mu'$ since, by assumption, $\mu$ is a partition of $r$. Thus, the last entry of $\wt(u)=\wt(t)$ is $|\mu|=:i$. By Figure \ref{figure directions to move for ires and iind}(c), we need to show that neither $\lambda$ nor $\mu$ have two adjacent columns of the form $\raisebox{4pt}{$\wc \vc$}$.  Notice that if this were the case for one of $\mu$ or $\lambda$, then Proposition \ref{prop common weight over arbitr p means same no of arrows} implies that the same is true for the other too: This is true for $\lambda$ if and only if column $\delta-i-1$ of $\tau d(\lambda)$ is of the form $\bigtimes$ which is the case if and only if column $\delta-i-1$ of $\tau d(\mu)$ is of the form $\bigtimes$ which is equivalent to $\mu$ having two adjacent columns of the form $\raisebox{4pt}{$\wc \vc$}$. But $|\lambda|=|\mu|=r$ implies $\lambda=\mu$ since the fact that $p>r$ means that the weight determines the partition in this case which shows the claim in the half integer case.
In the case when $\mu \in \Lambda^+(r)$, then by Figure \ref{figure directions to move for ires and iind}(d), we need to show that neither $\lambda$ nor $\mu$ have a column of the form $\bigtimes$. This, however, is immediate since if one has such a column, then by Proposition \ref{prop common weight over arbitr p means same no of arrows} the same column of the other is of the form $\bigtimes$, too. This again leads to $|\lambda|=|\mu|$ and thus $\lambda=\mu$ as in the previous case. Thus, Theorem \ref{thm ires of cells} implies that indeed $\ires \dmu= \dmup$ and $\ires \dla=\dlap$.

By construction $\wt(u')=\wt(t')$ so that by induction assumption, we must have $[\dlap:\lmup] \neq 0$.  This implies that there is a submodule $U$ of $\dlap$ with unique simple top $\lmup$: Let $U$ be a submodule of minimal composition length containing $\lmup$ as a composition factor. Using adjointness and the fact that $\ires \dla=\dlap$, we may deduce $$0 \neq \Hom(U,\dlap)=\Hom(U,\ires \dla)=\Hom(\iind U, \dla).$$ We will show that the part of the head of $\iind U$ which is not in the kernel of the implied non-zero homomorphism to $\dla$ is precisely $\lmu$. So assume $L(\tau)$ occurs in the head of $\iind U$ and is also a composition factor of $\dla$.  Then \begin{align} \label{eqn head of iind U}
0 \neq \Hom(\iind U, L(\tau))= \Hom(U, \ires L(\tau)).
\end{align} 
Since $L(\tau)$ is a composition factor of $\dla$, $\tau$ and $\lambda$ have a common JM weight by Proposition \ref{prop same cf means same wt}. Notice that because $[\dla:L(\tau)] \neq 0$, we can actually apply Proposition \ref{prop common jm weight implies certain form} even though $\tau$ is not a partition of $r$. This follows because the non-vanishing of the decomposition multiplicity over $P_d(\delta)$ also implies that the decomposition multiplicity does not vanish over $P_{|\tau|}(\delta)$ and $P_{|\tau+1/2|}(\delta)$, respectively, by Theorem \ref{thm form of dec matrix}. This in turn implies that $\lambda$ and $\tau$ also have a common JM weight over  $P_{|\tau|}(\delta)$ and $P_{|\tau+1/2|}(\delta)$, respectively, for which we can apply  Proposition \ref{prop common jm weight implies certain form}. Distinguish two cases:
\newline \textbf{Case 1:} $\lambda \in \Lambda^+(r+1/2)$. In this case, we must have $\lambda=\cdots \raisebox{4pt}{$\underset{\textstyle \raisebox{-1.9mm}{$\bigwedge$}}{?} \oc  \oc \cdots \oc \oc \vc$} \cdots$ and is identical to $\mu= \cdots  \raisebox{4pt}{$\qc \vc \oc  \cdots \oc  \wc \oc $}   \cdots$ outside the drawn columns. Proposition \ref{prop common jm weight implies certain form} now implies that $\tau$ must be obtained from $\lambda$  by  first moving the $\bigwedge$ in $\lambda$ to the right by $1$ then swapping the $\bigwedge$ with a neighbouring $\bigvee$ and then moving the $\bigwedge$ to the left by $1$ again. Since $|\tau|\geq |\lambda|$ ($L(\tau)$ is a composition factor of $\dla$), we can only move the $\bigwedge$ in $\lambda$ to the right which precisely yields $\mu$.  Thus, either $\tau=\lambda$ or $\tau=\mu$. Now both $\lambda$ and $\mu$ do not contain neighbouring columns of the form $\raisebox{4pt}{$\overset{\delta-i}{\wc}\vc$}$ as was shown earlier. Thus, we are not in Case $(f)$ of Theorem \ref{thm ires of simples} so that $\ires L(\tau)=L(\tau')$, for some partition $\tau'$, or zero. By (\ref{eqn head of iind U}), $\ires L(\tau)$ cannot be zero and since $U$ has simple head $L(\mu)$ we must moreover have $\ires L(\tau)=L(\mu')$. But $\mu'=\mu$ in this case so that the latter is false for $\tau=\lambda$ unless $\lambda=\lambda^-$ and $\mu=\lambda^+$ in the notation of Theorem \ref{thm ires of simples}(d). However, we have already shown the theorem to hold in this case so that we may assume $\tau=\mu$, as required.
\newline \textbf{Case 2:} $\lambda \in \Lambda^+(r)$. In this case, $\lambda=\cdots \raisebox{4pt}{$\wc \oc \raisebox{-4pt}{$\cdots$} \oc \vc$} \cdots$ and is identical to $\mu=\cdots \raisebox{4pt}{$\vc \oc \raisebox{-4pt}{$\cdots$} \oc \wc$} \cdots$ outside of the drawn columns. This time, partitions with a common JM weight with $\lambda$ are obtained by swapping a $\bigwedge$ with a neighbouring $\bigvee$. As before $|\tau| \geq |\lambda|$ so the only possibilities are $\tau=\lambda$ or $\tau=\mu$.  Because neither $\mu$ nor $\lambda$ has a  column of the form $\bigtimes$, we are not in Case $(e)$ of Theorem \ref{thm ires of simples} and by the same reasoning as in Case 1, we therefore have $\ires L(\tau)=L(\mu')$. But the latter is false for $\tau=\lambda$ unless $\mu=\cdots \raisebox{4pt}{$\vc \underset{i}{\wc}$} \cdots$ and $\lambda=\cdots \raisebox{4pt}{$\wc \underset{i}{\vc}$} \cdots$, see Figure \ref{figure directions to move for ires and iind}$(c)$. 
In this case, however, $\lambda=\lambda^-$ and $\mu=\lambda^+$  in the notation of Theorem \ref{thm ires of simples}$(c)$ and by Theorem \ref{thm differ by one box then common CF}, we therefore have $[\Delta(\lambda^+):\Delta(\lambda^-)] \neq 0$ which is what we want to show. 
 \end{proof}

\begin{rk}
 Notice that Proposition \ref{prop common jm weight implies certain form} gives a graphical description of when two partitions do have a common JM weight.
\end{rk}

The previous statement can also be interpreted in a Lie theoretic way. We will use a similar setup as in Section 8 of \cite{CDMwalledblocks}. So let $\{\epsilon_{0},\epsilon_{1},\epsilon_{2},\ldots\,\epsilon_r \}$ be a formal set of symbols, $X_r= \bigoplus_{i=0}^r \Z \epsilon_i$ be the weight lattice and let $\Phi=\{\pm(\epsilon_i-\epsilon_j) \ | \ i <j \}$ be the corresponding root system of type $A$. As usual we write a weight $\lambda=\sum_{i=0}^r \lambda_i \epsilon_i$ as $(\lambda_0,\ldots,\lambda_r)$. We define a weight to be dominant if  $\sum_{i=0}^r \lambda_i=0$ and for $i>0$ we have $\lambda_i \in \N_0$ with $\lambda_i\geq\lambda_{i+1}$. Thus the last $r$ entries form a partition and the first entry records the negative of the size of the partition. We denote by $X^+$ the set of dominant weights. We equip $X \otimes_{\Z} \R$  with an inner product by linear extension of $(\epsilon_i,\epsilon_j)=\delta_{ij}$ where $\delta_{ij}$ is the Kronecker Delta. The Weyl group $W$ is the group generated by reflections $s_{\alpha}$ for $\alpha \in \Phi$ where $s_{\alpha}=\lambda-\frac{2(\lambda,\alpha)}{(\alpha,\alpha)}\alpha$. $W$ is isomorphic to the symmetric group $S_{r+1}$ and acts on weights by permuting the entries. For the description of the blocks and decomposition numbers, we will need the dot action of $W$ on $X$ which just shifts the origin of the usual action by $\rho$. It turns out that we have to use different origins for the full integer and the half integer case. So if $d=r+1/2$, we set $\rho=(0,(1-\delta)-1,(1-\delta)-2,\ldots)$ and for $d=r$ we choose  $\rho=(1,(1-\delta)-1,(1-\delta)-2,\ldots)$. Then the dot action of $w \in W$ on $\lambda$ is given by $w \circ \lambda= w(\lambda+\rho)-\rho$. 

Notice that $\lambda+\rho$ contains precisely the data of the arrow diagram $d(\lambda)$ if $d \in \N$ and of $\tau d(\lambda)$ if $d \notin \N$. However,  the labels of the $\bigvee$ are shifted by $(1-\delta)$. Thus, we could redefine arrow diagrams as follows: The positions on the arrow diagram are labelled by integers and there is a $\bigwedge$ at position $(\lambda+\rho)_0$  and  $\bigvee$s at positions $(\lambda+\rho)_1,\ldots,(\lambda+\rho)_r$ (we can extend this to partitions with infinitely many parts by considering the entries to be zero from a certain point onwards which gives exactly the arrow diagrams). This leads to the following reformulation of Theorem \ref{thm determination of dec nos}.

\begin{theorem}
\label{thm lie theoretic det of dec nos}
Suppose $F=\C$ and $\lambda$ and $\mu$ are dominant (integral) weights. Then $[\Delta(\lambda):L(\mu)]=1$ if and only if there exists $i$ such that  $(1,i) \circ \mu=\lambda$. 
\end{theorem}

\begin{proof}
We know that  $[\Delta(\lambda):L(\mu)]=1$ if and only if $\lambda$ and $\mu$ have a common JM weight by Theorem \ref{thm determination of dec nos}. By Proposition \ref{prop common jm weight implies certain form}, this is the case if and only if $d(\lambda)$ and $d(\mu)$ are of a certain form. As remarked after Definition \ref{defn sigma action of arrow diag}, this form is equivalent to the arrow diagrams differing by the action of $\sigma$ or $\tau \inv \sigma \tau$, respectively. However, the action of $(1,i)$ is the same as the action of $\sigma$ and $\tau \inv \sigma \tau$ if $(1,i)\circ \mu$ is dominant: We must have that $(\lambda+\rho)_0=-|\lambda|$ (or $-|\lambda|+1$ in the full integer case) is between $(\lambda+\rho)_{i-1}$ and $(\lambda+\rho)_{i+1}$ and thus swapping the first and the $i$th entry amounts to swapping a $\bigvee$ with a neighbouring $\bigwedge$. In the half integer case, the shift $\tau$ is hidden in the different choice of $\rho$. 
\end{proof}

Theorem \ref{thm determination of dec nos} also allows us to settle the question of semisimplicity over an arbitrary field which does not seem to feature in the existing literature. Semisimplicity of the partition algebra was first studied by Martin and Saleur over $\C$ \cite{MartinSaleur} and also in \cite{HalversonRam} in the half integer case over $\C$ when $\delta \geq 2$.

\begin{theorem}
\label{thm semisimplicity}
The partition algebra $P_d(\delta)$ (with $d \in \{r,r+1/2 \}$) over the field $F$ of characteristic $p$ is semisimple if and only if $\delta \notin \{2(d-r),\ldots,2(d-1) \}$ (considered as a subset of $\Z 1_F$) and either $p=0$ or $p>r$.
\end{theorem}

\begin{proof} 
In the full integer case over a field of characteristic $0$, this follows from \cite{MartinSaleur}. In the half-integer case over $\C$  for $\delta \geq 2$, this was shown in Theorem 3.37 of \cite{HalversonRam}. Upon reduction modulo $p$ the algebra will, of course, remain non-semisimple if it was in characteristic $0$. Furthermore, if $p \leq r$, then the symmetric group algebra $FS_r$, which is a quotient of $P_d(\delta)$, will be non-semisimple. Thus it suffices to show that firstly the partition algebra is semisimple if $\delta$ is not in the prime subfield of $F$ or $\delta \notin \{2(d-r),\ldots,2(d-1) \}$ for both the cases $p=0$ and $p>r$, and secondly that the partition algebra is not semisimple if we are in the half-integer case with $p=0$ and $\delta=1$.

If $d=r+1/2$, $p=0$ and $\delta=1$, then $[\Delta(1^{r-1}):L(1^r)]=1$ by Theorem \ref{thm determination of dec nos}: We choose any $(1^r)$-tableau and extend it to tableaux $u$ and $t$ over $P_{r+1/2}(\delta)$ by adding the partitions $(1^r)$ and $(1^{r-1})$, respectively, at the end. Thus the last entries of the weight of $u$ and $t$ are $r$ and $\delta-(-r+1)=r$ so that $u$ and $t$ have the same weight. 

Now assume that $\delta$ is either not in the prime subfield or $\delta \notin \{2(d-r),\ldots,2(d-1) \}$.  $P_d(\delta)$ is semisimple if and only if all cell modules are irreducible. Assume for contradiction that $P_d(\delta)$ is not semisimple. Then there is some cell module $\Delta(\lambda)$ which has a composition factor $L(\mu)$ with $\mu \neq \lambda$. By Theorem \ref{thm form of dec matrix}, we may assume that $r=|\mu|$ and $|\lambda| < |\mu|$. Furthermore, by Theorem \ref{thm determination of dec nos} we know that $\lambda$ and $\mu$ have a common JM weight. Let $u$ and $t$ be  $\mu$- and $\lambda$-tableau, respectively, of the same weight. Since $u \neq t$, there is $k \in 1/2 \Z$ which is minimal such that $u^k \neq t^k$.  

Start with the case $k \in \N$. Then    $t^{k-1/2}=t^k$  and there is a box $\epsilon$ added to $u^k$. Since $u$ and $t$ have the same weight, we deduce from Theorem \ref{thm eigenvals of JM action} that   $\delta=c(\epsilon)+|t^{k}|$. If $\delta$ is not in the prime subfield then we get a contradiction straight away, since $c(\epsilon)$ and $|t^k|$ certainly are. And $c(\epsilon)$ must be in the range of $|u^{k-1/2}|=|t^{k-1/2}|$ (last box in the first row) and $-|u^{k-1/2}|=-|t^{k-1/2}|$ (last box in first column) so that $\delta$ is in the range of $2|t^{k-1/2}|$ and $0$. Since $|t^{k-1/2}| < |u^k| \leq r$, we deduce that if $P_d(\delta)$ is not semisimple then $\delta$ is in $\{0,\ldots,2r-2\}$ -- a contradiction unless $\delta=0$ and we are in the half-integer case and $\delta=0$. In this case, however, we know that the tableaux cannot end at step $k$. At step $k+1/2$, we must have $u^{k+1/2}=u^k$ (since $\mu$ is a partition of $r$) and $t^{k+1/2}=t^k-\epsilon'$ for some removable box $\epsilon'$. The latter follows since $L_{k+1/2}$ would act by $|t^{k}|$ on $t$ and $|u^k|$ on $u$ which implies $|u^k|=|t^k|$ since $u$ and $t$ have the same weight. Thus $\delta=|u^k|+c(\epsilon')=|t^{k-1/2}|+1+c(\epsilon')$. Since $\epsilon'$ is removable from $t^k=t^{k-1/2}$, its content is in the range $|t^{k-1/2}|-1$ and $-|t^{k-1/2}|+1$ so that $\delta$ is in the range $2|t^{k-1/2}|$ and $2$ -- contradiction to the fact that $\delta=0$. 

In the case $k \in 1/2\N$ one obtains that $\delta=|u^{k-1/2}|+c(\epsilon)$ where $\epsilon$ is a removable box of $t^{k-1/2}$ and hence in the range between $|t^{k-1/2}|-1$ and $-|t^{k-1/2}|+1$. Thus, $\delta$ must be between $2|t^{k-1/2}|-1$ and $1$ and since $|t^{k-1/2}| \leq r$, we get the required contradiction if $d=r+1/2$. By a similar argument as before, in the case $d=r$ and $\delta=2r-1$ there must be an additional step in the tableaux and we get $\delta=c(\epsilon') + |t^{k-1/2}|-1$ where $\epsilon'$ is  addable to $u^{k-1/2}$. Thus $\delta$ must be between $2|t^{k-1/2}|-2 \leq 2r-2$ and $0$ contradicting $\delta=2r-1$. 
\end{proof}

\section{The blocks of the partition algebra}
We will use the results from the previous sections to describe the blocks of the partition algebra over an arbitrary field thus recovering the results of \cite{BDK}. The ultimate goal is to use this description to categorify the blocks as weight spaces of an (affine) Weyl group action in the next section. If the parameter is an integer, then the partition algebra can be defined over the integers and we can employ the machinery of $p$-modular systems to reduce modulo $p$. Thus, if two simple modules $L(\lambda)$ and $L(\mu)$ are in the same block of $P_d(\delta+kp)$ in characteristic $0$ for some $k \in \Z$, then they are also in the same block over the partition algebra $P_d(\delta)$ over a field of characteristic $p$.   If we add to these block relations induced from characteristic $0$ the block relations coming from the symmetric group in characteristic $p$, then we get a sufficient condition for two modules to be contained in the same block. 

It  turns out that this condition is also necessary. In order to show this we use
a necessary condition for two simple modules to be in the same block given by Mathas \cite{Mathaspaper}. He shows in the more general context of cellular algebras which admit JM elements that two simple modules are in the same block only if they are in the same residue linkage class. Residue linkage is just the transitive closure of the property of having a common JM weight so that we have good control over the residue linkage classes. 

We will start with the definition of residue linkage:

\begin{defn}[\cite{Mathaspaper}]
Define an equivalence relation $\sim$ on $\Lambda^+(d)$ by taking the transitive closure of the following symmetric, reflexive relation: $\lambda \sim \theta$ if and only if $\lambda$ and $\theta$ have a common JM weight.
The equivalence classes are called residue linkage classes.
\end{defn}

Thus, we can state the necessary condition:

\begin{theorem}[Corollary 4.6 in \cite{Mathaspaper}] 
\label{thm necess block condition mathas}
Consider the partition algebra $P_d(\delta)$ over a field of characteristic $p>0$. Let $\lambda,\theta \in \Lambda^+(d)$ and suppose that $L(\lambda)$ and $L(\theta)$ are in the same block. Then  $\lambda \sim \theta$, that is $\lambda$ and $\theta$ are in the same residue linkage class.
\end{theorem}

\begin{proof}
We only have to show that the partition algebra satisfies the assumptions of Corollary 4.6 in \cite{Mathaspaper}. For this, we pick a minimal lift $\widehat{\delta} \in \{0,\ldots,p-1\}$ of $\delta$ (that is, reducing $\widehat{\delta}$ modulo $p$ gives $\delta$) and $(K,R,k)$ a $p$-modular system, that is  $K$ is the field of fractions of the commutative ring $R$ and $k$ is the residue field $R/\Pi$ where $\Pi$ is the unique maximal ideal of $R$. The characteristic of $K$ is $0$ and the characteristic of $k$ is $p$. Let $A_R$ be the partition algebra $P_d(\widehat{\delta}+2d*p)$, defined over $R$. Then $A_K$ is semisimple by Theorem \ref{thm semisimplicity} and hence the JM elements separate $A_K$, as required.
\end{proof}

We will now give a necessary condition for two partition to be in the same linkage class which follows immediately from Proposition \ref{prop common weight over arbitr p means same no of arrows}.

\begin{prop} \label{prop same linkage class implies same no of arrows} Suppose $\lambda, \theta \in \Lambda^+(d)$ are in the same residue linkage class. Then for all $k$ the $k$th columns of
\begin{enumerate}[(a)]
\item $d(\lambda)$ and $d(\theta)$ (if $d \in \N$)
\item  $\tau d(\lambda)$ and $\tau d(\theta)$ (if $d \notin \N$)
\end{enumerate}
contain the same number of arrows (that is, counting $\bigwedge$ and $\bigvee$).
\end{prop}

For the proof of the main result, we will need a graphical description of the block relations induced from characteristic $0$ and from the symmetric group in characteristic $p$. Recall that $\tau$ moves $\bigwedge$ up by $1$, see Definition \ref{defn of tau}. 

\begin{lem}
\label{lem graphical block relations}
Suppose $\lambda, \theta \in \Lambda^+(d)$ and that $d(\lambda)$ is obtained from $d(\theta)$ by one of the following operations:
\begin{enumerate}[(a)]
\item (symmetric group block relations)   By moving an arrow up or down by one within a column and then moving an arrow down or up, respectively, in a column. 
\item (full integer block relation) By moving a $\bigwedge$ in column $i$ to an arbitrary position in column $j$ and a $\bigwedge$ from column $j$ to column $i$ (valid if $d \in \N$).
\item (half integer block relation) By applying $\tau$, then performing operation $(b)$ and then applying $\tau \inv$ (valid if $d \notin \N$).
\end{enumerate}
Then $L(\lambda)$ and $L(\theta)$ are in the same block.
\end{lem}

\begin{proof}
Notice that the operation in Part $(a)$ simply removes a $p$-hook from a partition and then adds another $p$-hook so the initial and resulting partition have the same size and $p$-core. Thus the symmetric group simple modules $D(\lambda)$ and $D(\theta)$ are in the same block of the symmetric group of degree $|\lambda|=|\theta|$ and therefore also in the same block of the partition algebra by Theorem \ref{thm form of dec matrix}.

For Part $(b)$ and $(c)$, recall the definition of $\sigma$ from Definition \ref{defn sigma action of arrow diag}. It follows from Theorem \ref{thm determination of dec nos} and Proposition \ref{prop common jm weight implies certain form} that if $d \in \N$ and $\sigma d(\lambda)$ is non-zero, then $d(\lambda)$ and $\sigma d(\lambda)$ are in the same block. Similarly, if $d \notin \N$ and $\sigma \tau d(\lambda)$ is non-zero, then $d(\lambda)$ and $\tau \inv \sigma \tau d(\lambda)$ are in the same block. Composing the operation of swapping $\bigvee$ and $\bigwedge$ pairs and its inverse, we can deduce that changing $\bigvee \bigwedge$ pairs to $\bigwedge \bigvee$ pairs (and vice versa) preserves blocks even if the pairs are not necessarily neighbours, that is, might be separated by an arbitrary number of $\bigcirc$ and $\bigvee$, provided $d \in \N$. Analogously, we can for $d \notin \N$ apply $\tau$, change an arbitrary $\bigvee \bigwedge$ pair to a $\bigwedge \bigvee$ pair (and vice versa) and apply $\tau$. The resulting arrow diagram will still belong to the same block as the initial arrow diagram.

If two modules are in the same block in characteristic $0$, then their modular reductions are still in the same block in characteristic $p$. Let us study what the operations above imply for the abacus.  Suppose we have a $\bigvee$ in column $i$ of some runner of the abacus and a $\bigwedge$ in column $j$ (recall that columns are labelled  modulo $p$ with labels determined by the contents of the $\bigvee$ contained in the column). This is best illustrated graphically, see Figure  \ref{figure swap vee wedge pair on abacus}. 
\begin{figure}
\caption{Swapping a $\bigvee \bigwedge$ pair on the abacus.}
\label{figure swap vee wedge pair on abacus}
\subfigure[]{\hspace{0.5cm}\includegraphics[scale=0.17]{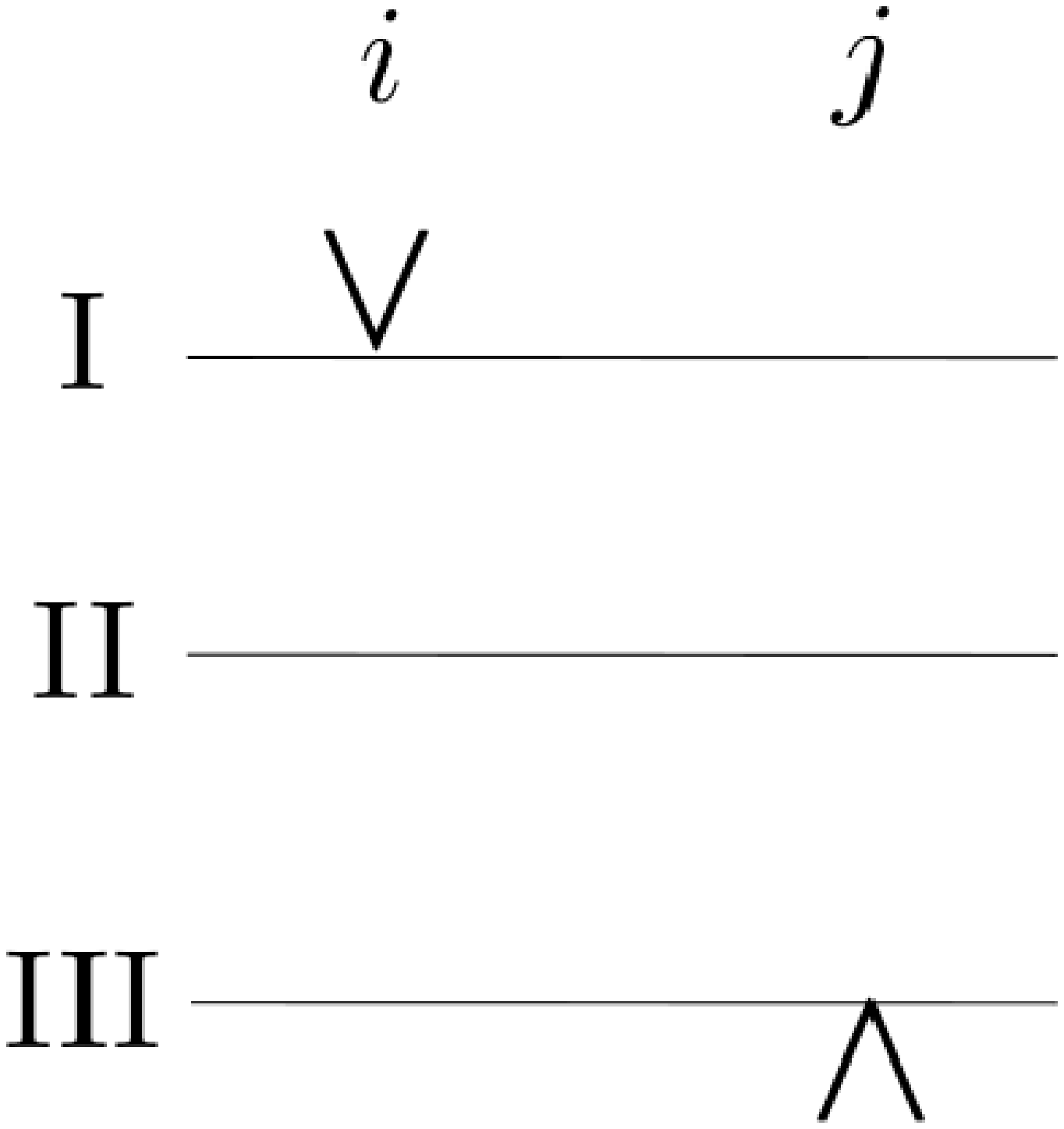}\hspace{0.5cm}}
\subfigure[]{\hspace{0.5cm}\includegraphics[scale=0.17]{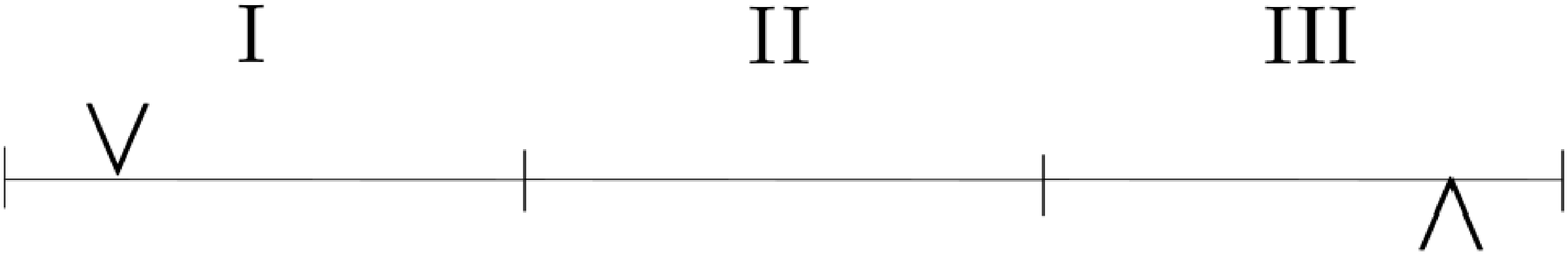}\hspace{0.5cm}}
\subfigure[  ]{\hspace{0.5cm}\includegraphics[scale=0.17]{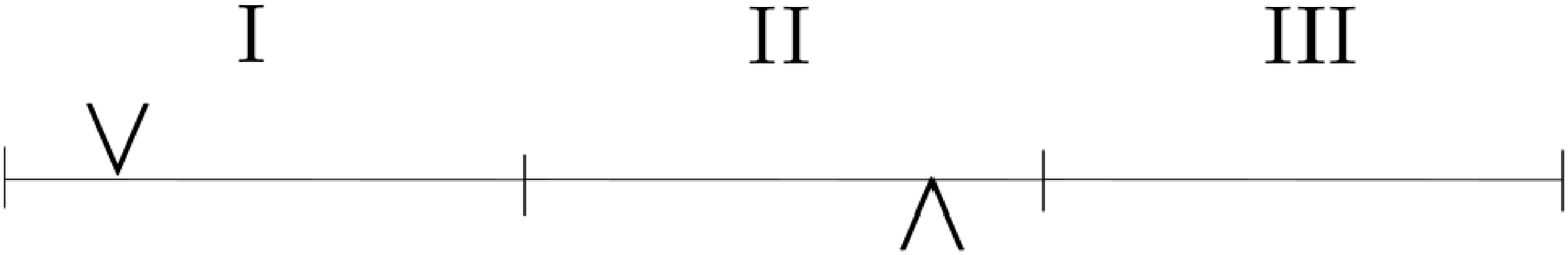}\hspace{0.5cm}}
\subfigure[  ]{\hspace{0.5cm}\includegraphics[scale=0.17]{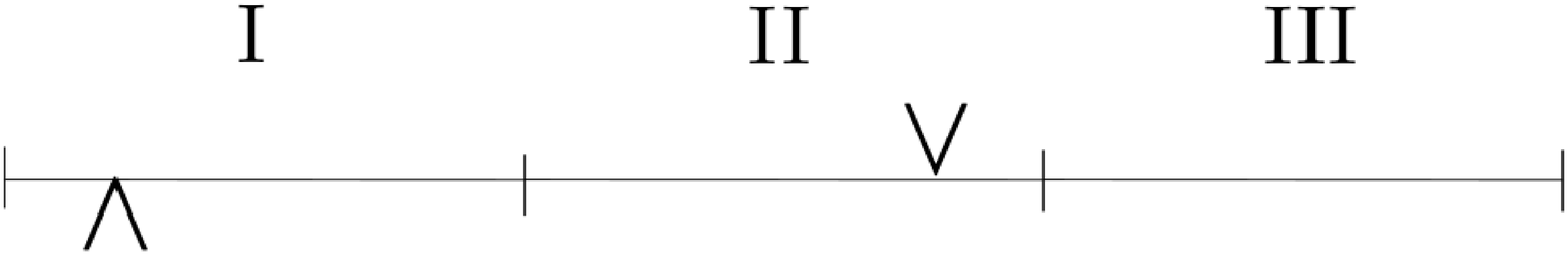}\hspace{0.5cm}}
\subfigure[  ]{\hspace{0.5cm}\includegraphics[scale=0.17]{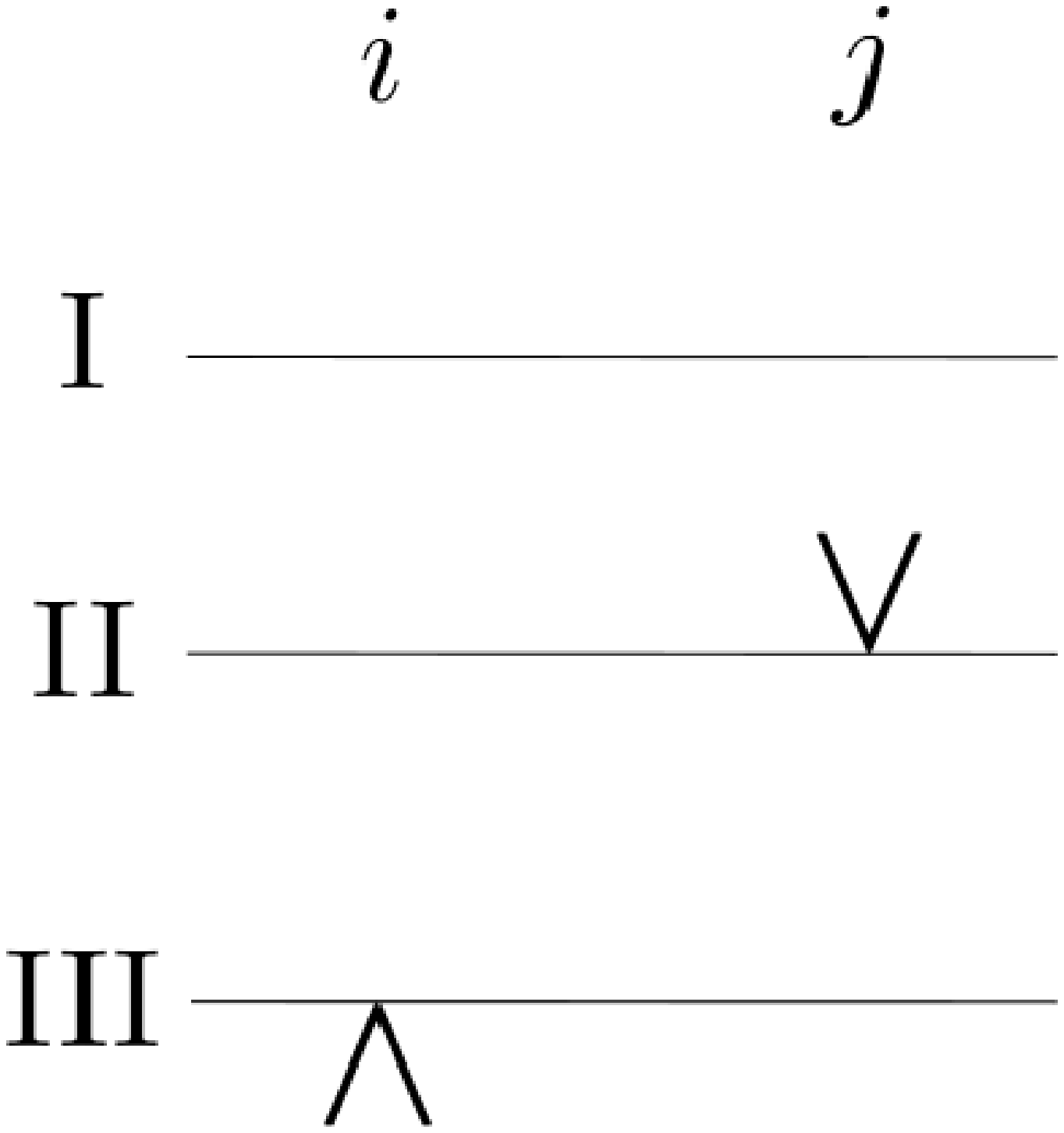}\hspace{0.5cm}}
\end{figure}

The initial situation is depicted in $(a)$ where we want to move the $\bigvee$ on the runner labelled I to column $j$ on the runner labelled II, say. If we draw the same partition on a characteristic $0$ diagram, then we get $(b)$, where the position of $\bigwedge$ depends on our choice of lift of the parameter. By an appropriate choice of lift of $\delta$, we can move $\bigwedge$ by $p$ steps to the left to get $(c)$. Now, we can change the $\bigvee \bigwedge$ pair to get $(d)$ and convert this back into a characteristic $p$ abacus diagram resulting in $(e)$ which is precisely of the claimed form. The general case follows completely analogously, as does the case when $d \notin \N$. 
\end{proof}

Thus, we can prove a full characterization of blocks:

\begin{theorem}
\label{thm determination of blocks}
Consider the partition algebra $P_d(\delta)$ over a field of arbitrary characteristic and let $\lambda, \theta \in \Lambda^+(d)$. Then the following are equivalent:
\begin{enumerate}[(a)]
\item $L(\lambda)$ is in the same block as $L(\theta)$.
\item $\lambda$ and $\theta$ are in the same linkage class. 
\item For all $k$, the $k$th columns of 
\begin{enumerate}[(i)]
\item ($d \in \N$) $d(\lambda)$ and $d(\theta)$.
\item ($d \notin \N$)  $ \tau d(\lambda)$ and $\tau d(\theta)$.
\end{enumerate}
contain the same number of arrows.
\end{enumerate}
\end{theorem}

\begin{proof} 
We have already shown that $(a)$ implies $(b)$ (Theorem \ref{thm necess block condition mathas}) and that $(b)$ implies $(c)$ (Proposition \ref{prop same linkage class implies same no of arrows}). We will show that $(c)$ implies $(a)$. Suppose $d(\lambda)$ and $d(\theta)$ or $\tau d(\lambda)$ and $\tau d(\theta)$ contain the same number of arrows in each column. We have to show that $L(\lambda)$ and $L(\theta)$ are in the same block. Given $d(\lambda)$ and $d(\theta)$, we can use Part $(a)$ of Lemma \ref{lem graphical block relations} to bring them to the following form: If $\bigwedge$ is in column $j$, then all $\bigvee$ outside of column $j$ are moved as far up as possible and in column $j$ all $\bigvee$ except for one are also moved as far up as possible at the expense of moving the lowest $\bigvee$ in column $j$ further down. Since the resulting partition will label a simple module which is still in the same block as $L(\lambda)$ and $L(\theta)$, we may assume that $d(\lambda)$ and $d(\theta)$ are of this form.

 Suppose $d \in \N$. Then we can use Part $(b)$ of Lemma \ref{lem graphical block relations}, to also move the lowest $\bigvee$ in column $j$ up as far as possible in both $d(\lambda)$ and $d(\mu)$  since Lemma \ref{lem graphical block relations} did not require $i \neq j$. In other words we may assume that $d(\lambda)$ and $d(\theta)$ are $p$-cores. 
 
 Suppose, the columns of $d(\lambda)$ and $d(\mu)$ containing  $\bigwedge$ are $i$ and $j$, respectively.  Then outside of  column $i$ and $j$, the arrow diagrams $d(\lambda)$ and $d(\theta)$ are identical as they contain the same number and same type of arrows in each column, namely only $\bigvee$.  If $i=j$, then we are done since column $i$ will contain the same number of $\bigvee$ in both $d(\lambda)$ and $d(\theta)$ and therefore the $p$-cores are identical. So assume this is not the case. Hence, in column $i$, $d(\lambda)$  contains  one less $\bigvee$ than $d(\lambda)$ (due to the presence of $\bigwedge$) and in column $j$ it contains one more $\bigvee$ than $d(\theta)$. Now again use Part $(b)$ of Lemma \ref{lem graphical block relations} to move the additional $\bigvee$ in column $j$ of $d(\lambda)$ to the furthest up free position of column $j$ and move the $\bigwedge$ from column $i$ to $j$. But this precisely yields $d(\theta)$ so that $\lambda$ and $\theta$ are in the same block, as required.

The case $d \notin \N$ is very similar and can also be obtained via the Morita equivalence in Theorem \ref{thm morita equiv of half with full}.
\end{proof}

We can also translate this description into a Lie theoretic one quite naturally, thus recovering the description of blocks in \cite{BDK} as unions of orbits of the action of an affine Weyl group.  In order to show this, recall the notation after Theorem \ref{thm determination of dec nos}. The affine Weyl group $W_p$ is generated by elements of the form $s_{\alpha,k}$ where $\alpha \in \Phi$ and $k \in \Z$ with action defined by $ s_{\alpha,k}\lambda=\lambda-((\lambda,\alpha)-kp)\alpha$ where $p$ is the characteristic of the field. The corresponding dot action is given as before by $w \circ \lambda=w(\lambda+\rho)-\rho$.

\begin{theorem}
\label{thm lie determination of blocks}
Consider the partition algebra $P_d(\delta)$ over a field of arbitrary characteristic and let $\lambda, \theta \in \Lambda^+(d)$. Then $L(\lambda)$ is in the same block as $L(\theta)$ if and only if $\lambda \in W_p \ \theta$, that is, $\lambda$ and  $\theta$ are contained in the same affine Weyl group orbit under the dot action.
\end{theorem}

\begin{proof} 
The affine Weyl group action is generated by the Weyl group actions for $P_d(\delta+kp)$ with $k \in \Z$. Each Weyl group is isomorphic to the symmetric group $S_{r+1}$ which is generated by elements of the form $(1,i)$ for $i=2,\ldots,r+1$. The action of $(1,i)$ was shown to be equivalent to the action of $\sigma$ in the proof of Theorem \ref{thm lie theoretic det of dec nos}. Therefore, the action of the affine Weyl group is generated precisely by the operations in Lemma \ref{lem graphical block relations}. In particular, it preserves blocks. Furthermore, the proof of Theorem 5.5 only used the relations given by Lemma \ref{lem graphical block relations}. Thus, if $L(\lambda)$ is in the same block as $L(\theta)$, we can use these relations (that is affine Weyl group actions) to convert $\lambda$ into $\theta$, as required.
\end{proof}

\section{A quantum group action on Fock space}
In this section we will show that we can define an analogue of the Fock space with a natural action of a quantum group of type $A$ in such a way that the weight spaces of the Grothendieck group of this Fock space are precisely the blocks of the partition algebra. The action of a quantum group will be similar in spirit to the case of the Fock space for Hecke algebras of type $A$ and will, in particular, be derived from the action of $i$-induction and $i$-restriction functors. There are, however, two notable differences to the classical setting. Firstly, because of the occurrence of half integer partition algebras, we will have to change the parameter $\delta$ after the action of the functors. This will require to move between partition algebras with different parameters and the definition of the Fock space is adjusted accordingly. Secondly, the induction functor for partition algebras is in general only right exact and therefore badly behaved when passing to Grothendieck groups. We will see that we can still define a quantum group action on Fock space, even though it will not be simply the induced action from the action of $i$-induction functors.

Recall that given a module category $C$ the Grothendieck group of $C$, denoted $[C]$, is the abelian group generated by the objects of $C$ modulo a relation $[E]=[M]+[N]$  whenever there is an exact sequence $0 \rightarrow M \rightarrow E \rightarrow N \rightarrow 0$. The complexified Grothendieck group $[C]_{\C}$ is obtained by extending scalars from $\Z$ to $\C$, that is $[C]_{\C}=[C] \otimes_{\Z} \C$. For convenience, we will from now on omit the subscript and denote by $[C]$ the complexified Grothendieck group. Given a functor $G$ on $C$, we define a map $[G]$ on $[C]$ by $[G]([M])=[G(M)]$ for any $M \in C$. We will also recall the definition of the quantum groups $\mathcal{U}_q(\mathfrak{sl_\infty})$ and $\mathcal{U}_q(\widehat{\mathfrak{sl}_p})$ by generators and relations. For $n$ and $k$ with $k \leq n$, we set $[n]_q=\frac{q^n-q^{-n}}{q-q \inv}$, $[n]_q^!=[1]_q [2]_q \cdots [n]_q$ and $\left[\begin{array}{c} n \\ k
\end{array} \right]_q= \frac{[n]_q^!}{[n-k]_q^! [k]_q^!}$.
\begin{defn}
Let  $\C(q)$ be the ring of rational functions in the indeterminate $q$ with complex coefficients. Let $I_p=\Z/p\Z$ and define  matrices $A_0=(a_{ij})_{i,j \in I_0}$ and $A_p=(a_{ij})_{i,j 
\in I_p }$ for $p$ prime  by setting $$ a_{ij}=\begin{cases} 2 & \textrm{if \ }  i=j \\ -1 & \textrm{if \ } |i-j|=1  \\ 0 & \textrm{otherwise}
\end{cases} \textrm{ \ or \ } a_{ij}=\begin{cases} 2 & \textrm{if \ }  i=j \\ -1 & \textrm{if \ } |i-j|=1 \textrm{ \ or \ }  |i-j|=p-1 \\ 0 & \textrm{otherwise}
\end{cases},$$ respectively.  Let $\mathfrak{g}$ be $\mathfrak{sl_\infty}$ (by which we mean the Lie algebra corresponding to the doubly infinite type $A$ case) if $p=0$ and $\widehat{\mathfrak{sl_p}}$ if $p>0$. Thus $A_p$ is the Cartan matrix of $\mathfrak{g}$. We also define $\mathfrak{h}_0$ to be the free $\Z$-module with basis $\{h_i \ | \ i\in \Z \}$ and $\mathfrak{h}_p$ to be the free $\Z$-module with basis $\{h_0,\ldots,h_{p-1},d \}$.  The quantum group
$\mathcal{U}_q(\mathfrak{g})$ is the $\C(q)$ algebra with generators $E_i,F_i,K_h^{\pm 1}$ for $i \in I_p$ and $h \in \mathfrak{h}$ with relations given by:
\begin{align}
\ & K_0=1, \ \   K_hK_g=K_{h+g} \label{eqn qgroup rel cartan} \\  \label{qgroup rel he and hf} \ & K_iE_j=q^{a_{ij}}E_jK_i, \ \  K_iF_j=q^{-a_{ij}}F_jK_i  \\
\ &  E_iF_i-F_iE_i=\delta_{ij} \frac{K_{h_i}-K_{-h_i}}{q-q^{-1}} \label{eqn qgroup rel defn of h}\\
 \ & \sum_{k=0}^{1-a_{ij}} (-1)^k \left[ \begin{array}{c} 1-a_{ij}  \\ k
  \end{array} \right]_q E_i^{1-a_{ij}-k}E_jE_i^k=0 \label{eqn quantum group rel serre for e} \\
 \ &  \sum_{k=0}^{1-a_{ij}} (-1)^k \left[ \begin{array}{c} 1-a_{ij} \\ k
  \end{array} \right]_q F_i^{1-a_{ij}	-k}F_jF_i^k=0 \label{eqn quantum group rel serre f} \\
\ &    K_dE_i=q^{\delta_{i0}}E_0K_d, \ \  K_dF_i=q^{-\delta_{i0}}F_iK_d  \label{eqn quantum group affine Kd}
 \end{align} where $i,j \in I_p$ $g,h \in \mathfrak{h}$, $K_{h_i}=K_i$ and the last relation only holds if $p>0$.
\end{defn}

The Fock space for partition algebras is defined as follows:

\begin{defn}
 Let $F$ be a field and denote by $P_r(\delta)\textrm{-mod}$ the category of finite dimensional left modules of the partition algebra over the ground field $F$. Let $C_{r,\delta}=\bigoplus_{i \in \N} P_i(r-i+\delta)\textrm{-mod}$. The Fock space $\mathcal{F}_{r,\delta}$ for the partition algebra is defined to be $\mathcal{F}_{r,\delta}=[C_{r,\delta}]$.
\end{defn}

Notice that $C_{r,\delta}$ and $\mathcal{F}_{r,\delta}$, respectively, contain the category of representations of $P_r(\delta)$ which justifies the notation. However, we will usually omit the subscripts if they are clear from the context.

We will now define endofunctors on $C_{r,\delta}$ which inspire the action of the quantum group when passing to the Fock space:

\begin{defn}
 We define two functors $e_{i,r,\delta}: P_r(\delta)\textrm{-mod} \to P_{r-1}(\delta-1)\textrm{-mod}$ and $f_{i,r,\delta}: P_r(\delta)\textrm{-mod} \to P_{r+1}(\delta+1)\textrm{-mod}$ by setting $e_{i,r,\delta}=T \circ \ires_r $ and $f_{i,r,\delta}=\iind_r \circ T \inv$ where $T$ is as in Theorem \ref{thm morita equiv of half with full}. We thus define endofunctors $e_i$ and $f_i$ on $C$ by setting $e_i=\displaystyle\bigoplus_{r  \in \N, \delta \in \Z} e_{i,r,\delta}$ and $f_i=\displaystyle\bigoplus_{r  \in \N, \delta \in \Z} f_{i,r,\delta}$.
\end{defn}

The functors $e_i$ and $f_i$ have implicitly already appeared in Proposition \ref{prop common weight over arbitr p means same no of arrows}. Roughly speaking they simply move arrows from column $i-1$ to $i$ and vice-versa and in this way mimic the action of $i$-induction and $i$-restriction functors for Hecke algebras of type $A$ on the  abacus. More precisely, they map cell modules to a cell filtered module and it follows from the proof of Proposition \ref{prop common weight over arbitr p means same no of arrows} that the graphical action on the labels of the cell modules is simply given as follows: $e_i$ and $f_i$, respectively, map a cell module $\Delta(\lambda)$ to a module filtered by cell modules whose labels are the same as $\lambda$ except that one arrow is moved from column $i$ to $i-1$ and $i-1$ to $i$, respectively.

The action of $U_q(\mathfrak{g})$ on $\mathcal{F}$ is defined in a very similar way as in classical case, see Theorem 6.13 in \cite{Mathasbook}:
\begin{defn}
\label{defn action of uqslp on F}
 Fix $r$ and $\delta$ and set $\mathcal{F}=\mathcal{F}_{r,\delta}$.
Let $[\Delta_k(\lambda)] \in \mathcal{F}$ which is a module for $P_k(r-k+\delta)$. Consider the arrow diagram $d(\lambda)$ corresponding to $\lambda$ and parameter $r-k+\delta$. We call an arrow of $d(\lambda)$ $i$-addable if it can be moved up by $1$ from column $i-1$ to column $i$. Similarly, an $i$-removable arrow is an arrow which can be moved from column $i$ to column $i-1$. 

 We denote by $A_i(\lambda)$ and $R_i(\lambda)$ the number of $i$-addable and $i$-removable arrows, respectively. If $\epsilon$ is an $i$-addable arrow, we denote by $A_i^a(\lambda,\epsilon)$/$A_i^b(\lambda,\epsilon)$ the number of $i$-addable arrows above/below $\epsilon$. Similarly, if $\epsilon'$ is an $i$-removable arrow, then we denote by $R_i^a(\lambda,\epsilon')$/$R_i^b(\lambda,\epsilon')$ the number of $i$-removable arrows above/below $\epsilon'$. We also denote by $N_d(\lambda)$ the number of arrows in column $0$ of label larger than $-k$. 

If $\epsilon$ is an $i$-addable/removable arrow, then $\lambda \pm \epsilon$ denotes the partition corresponding to the arrow diagram $d'$ which is obtained from $d(\lambda)$ by adding/removing $\epsilon$. Notice that $d'$ is an arrow diagram for the parameter $r-k+\delta \pm 1$ since we have only moved one arrow (usual addition/removal of boxes means to move a $\bigwedge$ and a $\bigvee$ to take into account the changed size of the partition). This is precisely the graphical realization of the functors $e_i$ an $f_i$. Here, we abuse notation slightly since $\lambda\pm \epsilon$ might be equal to $\lambda$ if $\epsilon$ is an arrow below the line. For $i \in I_p$ define

\begin{align*}
K_i [\Delta_k(\lambda)] & = q^{A_i(\lambda)-R_i(\lambda)} [\Delta_k(\lambda)] \\
E_i [\Delta_k(\lambda)]& = \sum_{\epsilon \textrm{ $i$-removable}} q^{A_i^a(\lambda,\epsilon)-R_i^a(\lambda,\epsilon)} [\Delta_{k-1}(\lambda-\epsilon)] \\
F_i [\Delta_k(\lambda)]& = \sum_{\epsilon \textrm{ $i$-addable}} q^{A_i^b(\lambda,\epsilon)-R_i^b(\lambda,\epsilon)} [\Delta_{k+1}(\lambda+\epsilon)].
\end{align*}
If $p>0$, we also define $$K_d [\Delta_k(\lambda)]  = q^{N_d(\lambda)} [\Delta_k(\lambda)].$$
\end{defn}

\begin{rk}
Notice that the action of $E_i$ is precisely the action of $[e_i]$ if $q=1$. Furthermore, using the exact sequences from Theorem \ref{thm ires of cells} for the definition of $[f_i]$, the action of $F_i$ is the same as $[f_i]$. However, $f_i$ is not exact so that the action of $[f_i]$ might depend on the choice of representative in the Grothendieck group.
\end{rk}
We will now show a theorem which is analogous to the result of Hecke algebras of type $A$ due initially to Hayashi \cite{Hayashi} which was later extended to quantum groups by Misra and Miwa \cite{MisraMiwa}:
\begin{theorem}
\label{thm fock space is uqslp module}
Suppose the characteristic of the ground field is $p$ and let $\mathfrak{g}$ be $\mathfrak{sl_\infty}$ if $p=0$ and $\widehat{\mathfrak{sl_p}}$ if $p>0$.  Then the action defined in Definition \ref{defn action of uqslp on F} equips $\mathcal{F}$ with the structure of an integrable $U_q(\mathfrak{g})$-module.
\end{theorem}

\begin{proof}
 The theorem will be deduced from the Hecke algebra case which is, for example, proved in Theorem 10.3 of \cite{Arikibook}. In order to do so, notice first that it is enough to check the relations on the basis of $\mathcal{F}$ consisting of the classes of cell modules. Furthermore we may identify a cell module $\Delta_k(\lambda)$ with its arrow diagram $d(\lambda)$. We have to be careful, however, since we can recover $k$ from $d(\lambda)$ only modulo $p$: The arrows above the runners determine a partition in a unique way which also determines the labels above the runners. The position of $\bigwedge$ must be at the label $|\lambda|$ and so determines the labels  below the runners which in turn determines the parameter $\delta'$ up to a multiple of $p$. As claimed, this determines $k$ since the parameter is $\delta'=r-k+\delta$ and $r$ and $\delta$ were fixed.

Now the transition to the Hecke algebra case is done by viewing the arrow diagram as a usual abacus as follows: We split the exceptional runner into two runners, the upper one containing the $\bigvee$s and the lower one containing $\bigwedge$. Furthermore, we relabel the lower of the two newly created runners by letting the labels run from $p$ to $2p-1$ increasing from right to left and the runners below the exceptional runner by adding $p$ to each label. There is one significant difference in that the number of beads on the abacus is usually finite. However, this does not affect the action since the infinitely many boxes of content less than $-k$ will not be movable and can thus be neglected. 

In the abacus case, adding a box of content $i$ corresponds to moving a bead from position $i-1$ to $i$ on a runner and removing a box of content $i$ amounts to moving a bead from position $i$ to $i-1$ on the runner.   Also, the exponents of $q$ in Definition \ref{defn action of uqslp on F}  coincide with the coefficients defined in the Hecke algebra case, except for the definition of $K_d$ which we will consider separately.  Thus, if the bead on the runner corresponding to the bottom of the exceptional runner does not leave this runner after application of $E_i, F_i$ and $K_i$, we can recover a usual arrow diagram from the abacus. In other words, the actions of $E_i, F_i$ and $K_i$ are identical to the classical setting in this case. If $p>0$, then by Remark \ref{rk choice of abacus labels non rigid} the labels of the columns in the arrow diagram is arbitrary. Thus we can shift the labels, if necessary. If $p \neq 2$, then we can assume by this argument that all relations which involve moving beads/arrows in at most three neighbouring columns hold, that is, we can assume $|i-j|>1$. In the case $p=2$, we may assume this for two neighbouring columns which is already sufficient.

Now consider the remaining cases. Relation (\ref{eqn qgroup rel cartan}) is easily seen to hold. In order to show that relations (\ref{qgroup rel he and hf})-(\ref{eqn quantum group rel serre f}) hold, it suffices to show that the action of the following elements commute for $|i-j|>1$: $K_i$ and $E_j$, $K_i$ and $F_j$, $E_i$ and $F_j$, $E_i$ and $E_j$, $F_i$ and $F_j$. Each of these is clear from the graphical description of the action: Since $i \neq j \pm 1$, moving arrows in column $i$ does not affect moving arrows in column $j$. Also the number of $i$-addable/$i$-removable arrows does not change after moving arrows into/out of column $j$ and vice-versa. Finally, if $p>0$, then (\ref{eqn quantum group affine Kd}) requires the weight of $[\Delta_k(\lambda)]$ and $E_0[\Delta_k(\lambda)]$ and $F_0[\Delta_k(\lambda)]$, respectively to differ by $q$ and $q^{-1}$, respectively. However, this is clear since $E_0$ will move an additional arrow into column $0$ and $F_0$ will move an arrow out of column $0$.
\end{proof} 

It turns out that if we compare the weight spaces of this $U_q(\mathfrak{g})$-action with the result on blocks obtained in Theorem \ref{thm determination of blocks}, then the descriptions agree. This is again similar to the case of Hecke algebras of type $A$:

\begin{theorem}
\label{thm decat of block is weight space} Suppose $M$ is a weight space of the $U_q(\mathfrak{g})$-module $\mathcal{F}_{r,\delta}$. Then there is $k \in \N$ and a block $B$ of $P_k(r-k+\delta)\textrm{-mod}$ such that $[B]=M$.
\end{theorem}

\begin{proof}
By Theorem \ref{thm determination of blocks}, $\Delta_k(\lambda)$ and $\Delta_k(\mu)$ are in the same block of $P_k(\delta')$ if and only if $d(\lambda)$ and $d(\mu)$ contain the same number of arrows in each column (ignoring the infinitely many rows which consist of $\bigvee$ only). We will show that the information contained in a weight of $[\Delta_k(\lambda)]$ is equivalent to knowing the number of arrows in each column of $d(\lambda)$ as well as $k$.

The weight certainly determines the difference between the number of $i$-addable arrows and $i$-removable arrows. We first claim that this number is equal to the number of arrows in column $i$ minus the number of arrows in column $i-1$ of $d(\lambda)$. Notice that even though both numbers are infinite, the difference makes sense since from a high enough runner onwards, there will be a $\bigvee$ in each column and we choose to ignore this part. To show the claim note that if there is an arrow in column $i-1$ which is not $i$-addable then there must be an arrow in the same row but in column $i$ which prevents the arrow being addable. This in turn implies that the arrow in column $i$ is not $i$-removable. Therefore, the number of arrows in column $i$ which are not $i$-removable is equal to the number of arrows in column $i-1$ which are not $i$-addable and since every arrow in column $i$ is either $i$-removable or not $i$-removable (and similarly for column $i-1$), the claim follows.

 Therefore, if $\Delta_k(\lambda)$ and $\Delta_k(\mu)$ are in the same block, then $K_0,\ldots,K_{p-1}$ will act by the same eigenvalue. But because both $\lambda$ and $\mu$ are partitions of at most $k$, all positions with label less than $-k$ in column $0$ will be occupied by a $\bigvee$. Therefore, $K_d$ also acts by the same eigenvalue on both.
 
It remains to show that the weight determines the number of arrows in each column as well as $k$. By the considerations above, the weight encodes the difference of the number of arrows between adjacent columns. Furthermore, if $p=0$, then eventually there will be exactly $1$ arrow in each column so we know the number of arrows in each column. If $p>0$ we know the number of arrows in column $0$ if we ignore infinitely many rows just containing $\bigvee$ by the action of $K_d$ (notice that some rows which just contain $\bigvee$ might be counted as well, but only finitely many). Thus, the information provided by the weight yields how many arrows are contained in each column. 

Now once we are given an arrow diagram (without labels) we can reconstruct the labelling and this labelling is unique by the same arguments as in the proof of Theorem \ref{thm fock space is uqslp module}. Because the labelling also determines the parameter and because of the definition of the Fock space, we can therefore reconstruct $k$ up to a multiple of $p$ from the arrow diagram. Furthermore, if we have two diagrams with the same number of arrows in each column (again ignoring the infinitely many $\bigvee$ above on which they agree), then both arrow diagrams must have the same labelling since we can convert them into another by adding/removing $p$-hooks (that is, move arrows along columns) and swapping a $\bigvee$ and a $\bigwedge$ as in Lemma \ref{lem graphical block relations}. In other words, the weight determines the labelling uniquely. In $p=0$, we are done so assume $p>0$. The action of $K_d$ gives us the number of arrows in column $0$ of content at least $-k$ and we also know that all positions of content less than $-k$ are occupied by a $\bigvee$. Thus, $K_d$ can distinguish $[\Delta_k(\lambda)]$ and $[\Delta_{k+lp}(\lambda)]$: If it acts by some $m \in \N$ on $[\Delta_k(\lambda)]$, then it will act by $m+l$ on $[\Delta_{k+lp}(\lambda)]$.
\end{proof}

The last theorem may be rephrased by saying that the cell module $\Delta_r(\lambda)$ is in the same block of $P_r(\delta)$ as $\Delta_r(\mu)$ if and only if $[\Delta_r(\lambda)]$ and $[\Delta_r(\mu)]$ have the same weight.

\bibliography{referencesDec13}      
\bibliographystyle{amsalpha}

\end{document}